\documentclass[12pt,psamsfonts]{amsart}
\usepackage[top=3cm,bottom=2cm,left=1.5cm,right=1.5cm]{geometry}
\usepackage{mathpazo}
\allowdisplaybreaks
\usepackage{amsaddr}
\usepackage{overpic}
\usepackage{amsmath,amsthm,amssymb,amscd,amsfonts,amsbsy}
\usepackage{color}
\usepackage{hyperref,url}
\usepackage{float}
\usepackage{hyperref}
\usepackage{enumerate}
\usepackage{enumitem}
\usepackage{mathrsfs}
\usepackage{mathrsfs}
\usepackage{amsthm}
\usepackage{amsmath}
\usepackage{amsfonts}
\usepackage{amssymb}
\usepackage{amsthm}
\usepackage{amsmath}
\usepackage{amsfonts}
\usepackage{amssymb}
\usepackage{mathtools}
\usepackage{mathrsfs}
\usepackage{comment}
\usepackage{graphics, setspace}
\usepackage[utf8]{inputenc}
\usepackage{listings,xcolor}
\usepackage{listings}

\newcommand{\mathsym}[1]{{}}
\newcommand{\unicode}[1]{{}}

\usepackage[normalem]{ulem}

\newcommand{\R}{\ensuremath{\mathbb{R}}}

\usepackage[toc,page]{appendix}
\usepackage{tikz-cd}

\newtheorem {theorem} {Theorem}

\newtheorem {proposition}{Proposition}
\newtheorem {corollary}{Corollary}
\newtheorem {lemma}{Lemma}

\newtheorem {conjecture}{Conjecture}

\def\R{\mathbb R}

\begin{document}

\title[Generalized Chazy differential equation]
{Existence of a cylinder foliated by periodic orbits\\ in the generalized Chazy differential equation}

\author[J. Llibre, D. D. Novaes and C. Valls]
{Jaume Llibre$^1$, Douglas D. Novaes$^2$, and Claudia Valls$^3$}

\address{$^1$ Departament de Matem\`{a}tiques, Universitat Aut\`{o}noma de Barcelona, 08193 Bellaterra, Barcelona, Catalonia, Spain}
\email{jllibre@mat.uab.cat}

\address{$^2$ Departamento de Matem\'{a}tica, Instituto de Matem\'{a}tica, Estat\'{i}stica e Computa\c{c}\~{a}o Cient\'{i}fica (IMECC), Universidade Estadual de Campinas (UNICAMP), \ Rua S\'{e}rgio Buarque de Holanda, 651, Cidade Universit\'{a}ria Zeferino Vaz, 13083-859, Campinas, SP, Brazil}
\email{ddnovaes@unicamp.br}

\address{$^3$Departamento de Matem\'atica, Instituto Superior T\'ecnico, Universidade de Lisboa, Av. Rovisco Pais 1049--001, Lisboa, Portugal}
\email{cvalls@math.tecnico.ulisboa.pt}

\subjclass[2010]{34C23, 34C25, 34C45, 34A36}

\keywords{Generalized Chazy differential equation, discontinuous differential system, continuous differential system, analytic differential system, periodic orbit}

\maketitle

\begin{abstract}
The generalized Chazy differential equation corresponds to the following two-parameter family of differential equations
\begin{equation*}\label{gcdeq}
\dddot x+|x|^q \ddot x+\dfrac{k |x|^q}{x}\dot x^2=0,
\end{equation*}
which has its regularity varying with $q$ , a positive integer. Indeed, for $q=1$ it is discontinuous on the straight line $x=0$, whereas for $q$  a positive even integer it is polynomial, and for $q>1$ a positive odd integer it is continuous but not differentiable on the straight line $x=0$. In 1999, the existence of periodic solutions in the generalized Chazy differential equation was numerically observed for $q=2$ and $k=3$. In this paper, we prove analytically the existence of such periodic solutions. Our strategy allows to establish sufficient conditions  ensuring that the generalized Chazy differential equation, for $k=q+1$ and any positive integer $q$ , has actually an invariant topological cylinder foliated by periodic solutions in the $(x,\dot x,\ddot x)$-space . In order to set forth the bases of our approach, we start by considering $q=1,2,3$, which are representatives of the different classes of regularity.  For an arbitrary positive integer $q$ , an algorithm is provided for checking the sufficient conditions for the existence of such an invariant cylinder, which we conjecture that always exists. The algorithm was successfully applied up to $q=100$.
\end{abstract}


\bigskip

\noindent{\bf
In 1999, Géronimi et al  \cite{GFL} introduced a two-parameter family of third order differential equations called generalized Chazy differential equation. An interesting feature of the generalized Chazy differential equation is that its regularity changes with a positive integer parameter $q$. Indeed, for $q=1$ the generalized Chazy differential equation is discontinuous at $x=0$, whereas for $q$  a positive even integer it is polynomial, and for $q>1$ a positive odd integer it is continuous but not differentiable at $x=0$. They performed numerical computations and, under some constraints, they observed (only numerically) the existence of periodic solutions. Usually, to prove analytically the existence of periodic solutions in a given differential equation is not an easy problem. The difficulty increases significantly when the dimension or the order is greater than two.  In this paper, we develop a strategy to detect analytically the existence of periodic solutions in the generalized Chazy differential equation. Our strategy allows to establish sufficient conditions  ensuring that the generalized Chazy differential equation, for any positive integer $q$ , has actually an invariant topological cylinder foliated by periodic solutions in the $(x,\dot x,\ddot x)$-space. First,  we will focus our analysis in the cases $q=1,2,3$ (representatives of the above different classes of regularity), for which we will prove  the existence of such an invariant cylinder.  These initial cases will set forth the bases for an algorithmical approach when $q$  is an arbitrary given positive integer, which will be successfully applied up to $q=100$.}

\section{Introduction and statements of the main results}

In 1997, Feix et al. \cite{Fe} introduced the following general third order ordinary differential equation
\begin{equation}\label{e1}
\dddot x+ x^{3q+1}f(a,b)=0, \ \ \mbox{where $a=\dfrac{\dot x}{x^{q+1}}$ and $b=\dfrac{\ddot x}{x^{2q+1}}$},
\end{equation}
which are invariant under time translation and  rescaling symmetries.  By taking $f(a,b)=ka^2+b$, the differential equation \eqref{e1} becomes
\begin{equation}\label{e2}
\dddot x+ x^q \ddot x+ k x^{q-1}\dot x^2=0,
\end{equation}
which, for $q=1$ and $k=-3/2,$ corresponds to the well known Chazy differential equation, introduced by Chazy \cite{Ch} in 1911.

Numerical computations on the actual Chazy differential equation did not detect any periodic solution, while periodic solutions were numerically observed for a large number of initial conditions in the differential equation \eqref{e2} for $q = 2$ and $k = 3$. Those numerical observations were reported in 1999 by G\'eronimi et al. \cite{GFL}, which led them to introduce the so-called {\it generalized Chazy differential equation}
\begin{equation}\label{gcdeq}
\dddot x+|x|^q \ddot x+\dfrac{k |x|^q}{x}\dot x^2=0.
\end{equation}

Usually, to prove analytically the existence of periodic solutions in a given differential equation is not an easy problem. The difficult of the problem increases significantly when the dimension or the order is greater than two. The objective of this paper is double, first to provide an analytic proof of the existence of the periodic solutions detected only numerically in \cite{GFL}  in the generalized Chazy differential equation, and second to provide sufficient conditions ensuring that the generalized Chazy differential equation has, actually, an invariant cylinder foliated by periodic solutions in the $(x,\dot x,\ddot x)$-space.

The generalized Chazy differential equation can be written as the following first order differential system in $\R^3$
\begin{equation}\label{gcds}
X:\left\{\begin{aligned}
\dot x=&y,\\
\dot y=&z,\\
\dot z=&-|x|^q z+\dfrac{k |x|^q}{x}y^2.
\end{aligned}\right.
\end{equation}
When $k=q+1$ system \eqref{gcds} has the first integral
\begin{equation}\label{first}
H(x,y,z)=x z-\dfrac{y^2}{2}+|x|^q x y.
\end{equation}

Notice that the regularity of the generalized Chazy differential system \eqref{gcds} changes with $q$. Indeed, for $q=1$ the generalized Chazy differential system \eqref{gcds} is discontinuous on the straight line $x=0$, whereas for $q$  a positive even integer it is polynomial, and for $q>1$ an odd positive integer it is continuous but not differentiable on the straight line $x=0$.

\subsection{Main results} In this paper, we will propose an algorithmical approach to detect analytically the existence of an invariant cylinder foliated  by periodic orbits for any given positive integer $q$. But first, in order to set forth the bases of our approach, we shall focus our analysis in the cases $q=1,2,3,$ which are representatives of the above classes of regularity.

Our first main result is the following.

\begin{theorem}\label{main}
For $q=1,2,3$ and $k=q+1$, the generalized Chazy differential system \eqref{gcds} has an invariant cylinder foliated by periodic orbits.
\end{theorem}

Theorem \ref{main} is proven in Section \ref{sec:proofU1}. Its proof is divided into three subsections, which cover the cases $q=1$ (discontinuous), $q=2$ (polynomial), and $q=3$ (continuous), respectively. The proof follows by showing the existence of a fixed point of a suitable Poincar\'{e} return map in each negative energy level of the first integral $H$. Due to the nonlinear nature of the generalized Chazy differential system \eqref{gcds}, it is not possible to integrate the flow for obtaining an explicit expression of the Poincar\'{e} return map, thus tools from the qualitative theory must be employed.

In what follows, we provide the idea of the proof. First, by applying a rescaling in the variables and in the time, we reduce the analysis to the energy level $H=-1$, which is a topological cylinder. Second, by taking advantage of a symmetry of the system in this level, we show that a Poincar\'{e} return map is well defined, for which we can ensure the existence of a fixed point via the construction of a convenient trapping region. This implies the existence of a periodic orbit in each negative energy level, which varies smoothly with the energy, producing then a topological cylinder foliated by periodic orbits.

The above strategy, developed for proving the existence of invariant cylinders foliated by periodic orbits for $q=1,2,3$ can be extended for other positive integer values of $q$. Nevertheless, we have to make the following consideration upon this strategy. At some point in the proof of Theorem \ref{main}, it  will be  necessary to estimate the number of roots of some polynomials on specific intervals. While for $q=1,2,3$, it can be done with a Sturm procedure (see, for instance, \cite[Theorem 5.6.2]{SB}), for an undetermined $q$  the Sturm procedure does not work well. Alternatively, in what follows, our second main result provides sufficient conditions ensuring the existence of an invariant cylinder foliated by periodic orbits of the generalized Chazy differential system \eqref{gcds} when $q$ is an arbitrary given positive integer.

\begin{theorem}\label{main-general}
Let $q$  be a positive integer and consider the following polynomials on the variable $u$:
\begin{equation}\label{polynomials}
\begin{aligned}
P_0(u)=&8 q^2 - 4 q^2 (1 + 4 q) u +
 8 q^3 u^2-
 16 (1 + q)^2 u^{2(q-1)} +
 8 (1 + q) (3 + 2 q) (1 + 4 q) u^{2q-1} \\
 &  - (1 + 2 q) (9 + 2 q (7 + 4 q)^2) u^{2q}+
 8 q (1 + 4 q) (4 + q (5 + 2 q)) u^{2q+1}\\
&  -
 4 q^2 (7 + 4 q (2 + q)) u^{2(q+1)}  +
 8 (1 + q) (1 + 4 q) u^{4q+1}  - 16 q (1 + q) u^{2 (1+ 2 q)},\\
P^+ (u)=&2^{\frac{5+4q}{2(1+q)}} q (1 + q)^{\frac{1}{1+q}} (2 + q) -
 2^{\frac{1}{1+q}}(1 + q)^{\frac{2}{1+q}} (3 + 2 q)^2 u - 8 \sqrt{2} q u^q \\
 & +
 2^{\frac{4+3q}{2(1+q)}} (1 + q)^{\frac{2+q}{1+q}} (9 + 2 q) u^{1+q} - 2 (9 + 10 q) u^{1+2q} \\
 &  +
2^{\frac{5+4q}{2(1+q)}} q (1 + q)^{\frac{1}{1+q}} u^{2(1+q)} -
 4 \sqrt{2} q u^{2+3q} ,\\
P^- (u)=&2^{\frac{5+4q}{2(1+q)}} q (1 + q)^{\frac{1}{1+q}} (2 + q) -
 2^{\frac{1}{1+q}}(1 + q)^{\frac{2}{1+q}} (3 + 2 q)^2 u - 8 (-1)^q \sqrt{2} q u^q \\
 & \phantom{\le} +
 (-1)^q 2^{\frac{4+3q}{2(1+q)}} (1 + q)^{\frac{2+q}{1+q}} (9 + 2 q) u^{1+q} - 2 (9 + 10 q) u^{1+2q} \\
 & \phantom{\le} +
2^{\frac{5+4q}{2(1+q)}} q (1 + q)^{\frac{1}{1+q}} u^{2(1+q)} -
 4 (-1)^{q} \sqrt{2} q u^{2+3q}.
\end{aligned}
\end{equation}
Assume that the following conditions hold:
\begin{itemize}
\item[{\bf C1.}] the polynomial $P_0$ does not vanish on $I_0:=\big(2\,,\,(1+4q)/(2q)\big)$;
\item[{\bf C2.}] the polynomial $P^+ $ has  at most one root (counting its multiplicity) in $I^+ :=\big(0\,,\, u_{i,q}^D\big)$;
\item[{\bf C3.}] the polynomial $P^- $ has at most one root (counting its multiplicity) in $I^-:=(-2,0)$.
\end{itemize}
Then, for $k=q+1,$ the generalized Chazy differential system \eqref{gcds} has an invariant cylinder foliated by periodic orbits.
\end{theorem}

Theorem \ref{main-general} is proved in Section \ref{sec:q-arbitrary}. Notice that it allows the development of an algorithmical approach based on the Sturm procedure to detect analytically the existence of an invariant cylinder foliated by periodic orbits  of the generalized Chazy differential system \eqref{gcds} when $q$ is an arbitrary given positive integer. A Mathematica algorithm is provided as the supplementary material. We have ran the algorithm for $q=1,\ldots,100$, which provided the following result:

\begin{corollary}
For $k=q+1$ and $1\le q \le 100$, the generalized Chazy differential system \eqref{gcds} has an invariant cylinder foliated by periodic orbits.
\end{corollary}

The obtained results lead us to make the following conjecture:

\begin{conjecture}\label{main.bis}
For $k=q+1$ and $q$  a positive integer, the generalized Chazy differential system \eqref{gcds} has an invariant cylinder foliated by periodic orbits.
\end{conjecture}

Of course, the proof of Conjecture  \ref{main.bis} relies only on checking whether conditions {\bf C1}, {\bf C2}, and {\bf C3} of Theorem \ref{main-general} hold for any positive integer $q$.

\subsection{Structure of the paper} In Section \ref{sec:setup}, we provide the common bases for the proofs of our main results (Theorems \ref{main} and~\ref{main-general}), by reducing the problem to study the transition maps of planar vector fields. In Section \ref{sec:proofU1}, we present the proof of Theorem \ref{main}. Some details of this proof, concerning the Sturm procedure, will be provided in the Appendix. In Section \ref{sec:q-arbitrary}, we present the proof of Theorem \ref{main-general}. A Mathematica algorithm, based on a Sturm procedure, for checking the conditions of Theorem \ref{main-general} is provided as supplementary material.

\section{Reduced problem}\label{sec:setup}

Consider the hyperbola $\widetilde \Sigma=\{(u,v):\, u v+1=0\}$ and its positive and negative branches
\[
\widetilde \Sigma^+=\{(u,v):\,u v+1=0,\, u>0\} \ \text{and} \
\widetilde \Sigma^-=\{(u,v):\,u v+1=0,\, u<0\},
\]
respectively. In this section, we shall see that the problem of showing the existence of an invariant cylinder foliated by periodic orbits of the $3$D Chazy differential system \eqref{gcds} is equivalent to show the existence of a point $p$ in $\widetilde \Sigma^+$ whose flow through a planar differential system (associated to the restriction of \eqref{gcds} to the energy level $H=-1$) intersects $\widetilde \Sigma^-$ at $-p$.

Consider the first integral \eqref{first}. By solving $H(x,y,z)=-\omega^2$, with $\omega>0$, in the variable $y$, we obtain a family of  invariant surfaces $S_{\omega,q}:=S_{\omega,q}^-\cup S_{\omega,q}^+$, where
\[
S_{\omega,q}^{\pm}=\{(x,y,z)\in\R^3:\, y=f^{\pm}_{\omega,q}(x,z)\,\,\text{and}\,\, x z+\omega^2\geq0\},
\]
with
\[
f^{\pm}_{\omega,q}(x,z)=x \lvert x \rvert^q \pm\sqrt{x^{2(1+q)}+2xz+2\omega^2}.
\]
Notice that, for each $\omega$, $S_{\omega,q}^-$ and $S_{\omega,q}^+$ intersect the plane $y=0$ on the hyperbola
\[
\Sigma_{\omega}=\{(x,y,z):\,x z+\omega^2=0\}
\]
and, therefore, $S_{\omega,q}$ is a topological cylinder. The basis of our approach consists in proving that the invariant surface $S_{\omega,q}$ contains a periodic orbit $\gamma_{\omega,q}$ for each $\omega$.

In order to do that, we will define a Poincar\'{e} return map on the branches of the hyperbola $\Sigma_{\omega}^{\pm}$, namely,
\[
\Sigma_{\omega}^-=\{(x,y,z):\,x z+\omega^2=0,\, x<0\}, \quad \Sigma_{\omega}^+=\{(x,y,z):\,x z+\omega^2=0,\, x>0\}.
\]

Denote by $X_{\omega,q}^-$ and $X_{\omega,q}^+$ the reduced systems of \eqref{gcds} on $S_{\omega,q}^-$ and $S_{\omega,q}^+$, respectively, which are given by
\[
X_{\omega,q}^{\pm}:
\left\{
\begin{aligned}
\dot x=&x \lvert x  \rvert^q \pm\sqrt{x^{2(1+q)}+2xz+2\omega^2},\\
\dot z=&\frac{\lvert x  \rvert^q}{x} (-x z-(q+1) \left(x \lvert x  \rvert^q \pm\sqrt{x^{2(1+q)}+2xz+2\omega^2}\right)^2,
\end{aligned}\right.
\]
for $xz+\omega^2\geq0$. Let $\phi^\pm_{\omega,q}(t,\cdot)$ denote the flow of the vector field $X_{\omega,q}^\pm$.

In order to conclude the existence of a periodic orbit contained in $S_{\omega,q}$, it is sufficient to show the existence of a point $p_{0,q}\in \Sigma^-_{\omega,q}$ and $t_{0,q},t_{1,q}>0$ such that
\begin{equation}\label{fundrelation}
p_{1,q}:=\phi^-_{\omega,q}(t_{0,q},p_{0,q})\in \Sigma^+_{\omega}\ \text{and} \ \phi^+_{\omega,q}(t_{1,q},p_{1,q})=p_{0,q}.
\end{equation}
In this case the periodic orbit is given by $\gamma_{\omega,q}=\gamma^+_{\omega,q}\cup \gamma^-_{\omega,q},$ where
\[
\begin{aligned}
\gamma^-_{\omega,q}=&\{(x,y,z):\, (x,z)=\phi_{\omega,q}^-(t,p_{0,q}),\, y=f^-_{\omega,q}(\phi_{\omega,q}^-(t,p_{0,q})),\,\,\text{and}\,\, 0\leq t\leq t_{0,q}\}\,\,\text{and}\\
\gamma^+_{\omega,q}=&\{(x,y,z):\, (x,z)=\phi_{\omega,q}^+(t,p_{1,q}),\, y=f^+_{\omega,q}(\phi_{\omega,q}^+(t,p_{1,q})),\,\,\text{and}\,\, 0\leq t\leq t_{1,q}\}.
\end{aligned}
\]

Now, in order to simplify the reduced systems $X^{\pm}_{\omega,q}$, we apply the following change of variables and time-rescaling
\[
(x,z)=(\omega^{\frac{1}{1+q}} u,\omega^{\frac{1+2q}{1+q}}v) \ \text{and} \  t=\omega^{-\frac{q}{1+q}}\tau.
\]
Thus, the vector fields $X_{\omega,q}^\pm$ become
\begin{equation}\label{eq:vector}
\widetilde X_q^{\pm}:
\left\{
\begin{aligned}
u'=&u \lvert u \rvert^q \pm\sqrt{u^{2(1+q)}+2uv+2},\\
v'=&\frac{\lvert u \rvert^q}{u}\left (-u v-(q+1) \left(u \lvert u \rvert^q \pm\sqrt{u^{2(1+q)}+2uv+2}\right)^2\right),
\end{aligned}\right.
\end{equation}
for $u v+1\geq0$, and the curves $\Sigma_{\omega}^+$ and $\Sigma_{\omega}^-$ become $\widetilde \Sigma^+$ and $\widetilde \Sigma^-$, respectively.

Notice that, if $\phi_q^\pm(t,\cdot)$ denote the flow of the vector fields $\widetilde X_q^\pm$, then
\[
\phi^{\pm}_{q,\omega}(t,p)=\Omega\, \phi_q^{\pm}(\omega^{\frac{q}{1+q}}t,\Omega^{-1} p)\,\,\text{where}\,\, \Omega=\left(\begin{array}{cc}\omega^{\frac{1}{1+q}}&0\\
0&\omega^{\frac{1+2q}{1+q}}\end{array}\right).
\]

The existence of a periodic orbit will follow by showing the existence of a point $p_q^*=(u_q^*,-1/u_q^*)\in \widetilde\Sigma^+$ and a time $t_q^*>0$ such that
\begin{equation}\label{symcond}
\phi_q^-(t_q^*,p_q^*)=-p_q^*=(-u_q^*,1/u_q^*)\in \widetilde\Sigma^-.
\end{equation}
Indeed the vector fields $\widetilde X_q^{+}$ and $\widetilde X_q^{-}$ satisfy $\widetilde X_q^{+}(u,v)=-\widetilde X_q^{-}(-u,-v)$. This means that $\phi_q^+(t,p)=-\phi_q^-(t,-p)$ and, therefore, $\phi_q^+(t_q^*,-p_q^*)=-\phi_q^-(t_q^*,p_q^*)=p_q^*$. Hence, relationship \eqref{fundrelation} will hold by taking $$p_{0,q}=\Omega \, p_q^* \ \text{and} \ t_{0,q}=t_{1,q}=\omega^{\frac{1+q}{q}}t_q^*.$$

\section{Proof of Theorem \ref{main}}\label{sec:proofU1}

This section is devoted to the proof of Theorem \ref{main}, which is divided into three subsections, which cover the cases: $q=1$ (discontinuous), in Subsection \ref{sec:q1}; $q=2$ (polynomial), in Subsection \ref{sec:q2}, and $q=3$ (continuous), in Subsection \ref{sec:q3}. Some details concerning the Sturm procedure (see, for instance, \cite[Theorem 5.6.2]{SB}) will be provided in the Appendix.

\subsection{ Proof of Theorem \ref{main} for the discontinuous case $q=1$.}\label{sec:q1}  In this case, it follows from \eqref{eq:vector} that
\[
\widetilde X_1^{\pm}=\left\{\begin{array}{l} \widetilde X_{1,L}^{\pm}\quad \text{if}\quad u<0,\vspace{0.2cm}\\ \widetilde X_{1,R}^{\pm}\quad \text{if}\quad u>0, \end{array}\right.
\]
for $u v+1\geq0$, where
\[
\widetilde X_{1,L}^{\pm}:
\left\{
\begin{aligned}
u'=&-u^2 \pm\sqrt{u^{4}+2uv+2},\\
v'=&u v+2 \left(u \lvert u \rvert \pm\sqrt{u^{4}+2uv+2}\right)^2,
\end{aligned}\right.
\]
and
\[
\widetilde X_{1,R}^{\pm}:
\left\{
\begin{aligned}
u'=&u^2 \rvert \pm\sqrt{u^{4}+2uv+2},\\
v'=&-u v-2 \left(u \lvert u \rvert \pm\sqrt{u^{4}+2uv+2}\right)^2.
\end{aligned}\right.
\]
See Figure \ref{f1} for an idea of the phase space of the vector fields $\widetilde X_{1}^-$ (in the left) and  $\widetilde X_{1}^+$ (in the right). Here, the Filippov's convention \cite{Filippov88} is assumed for the trajectories of $\widetilde X_1^{\pm}$.

\begin{figure}[H]
\begin{minipage}{0.45\linewidth}
\centering
\includegraphics[width=5.5cm]{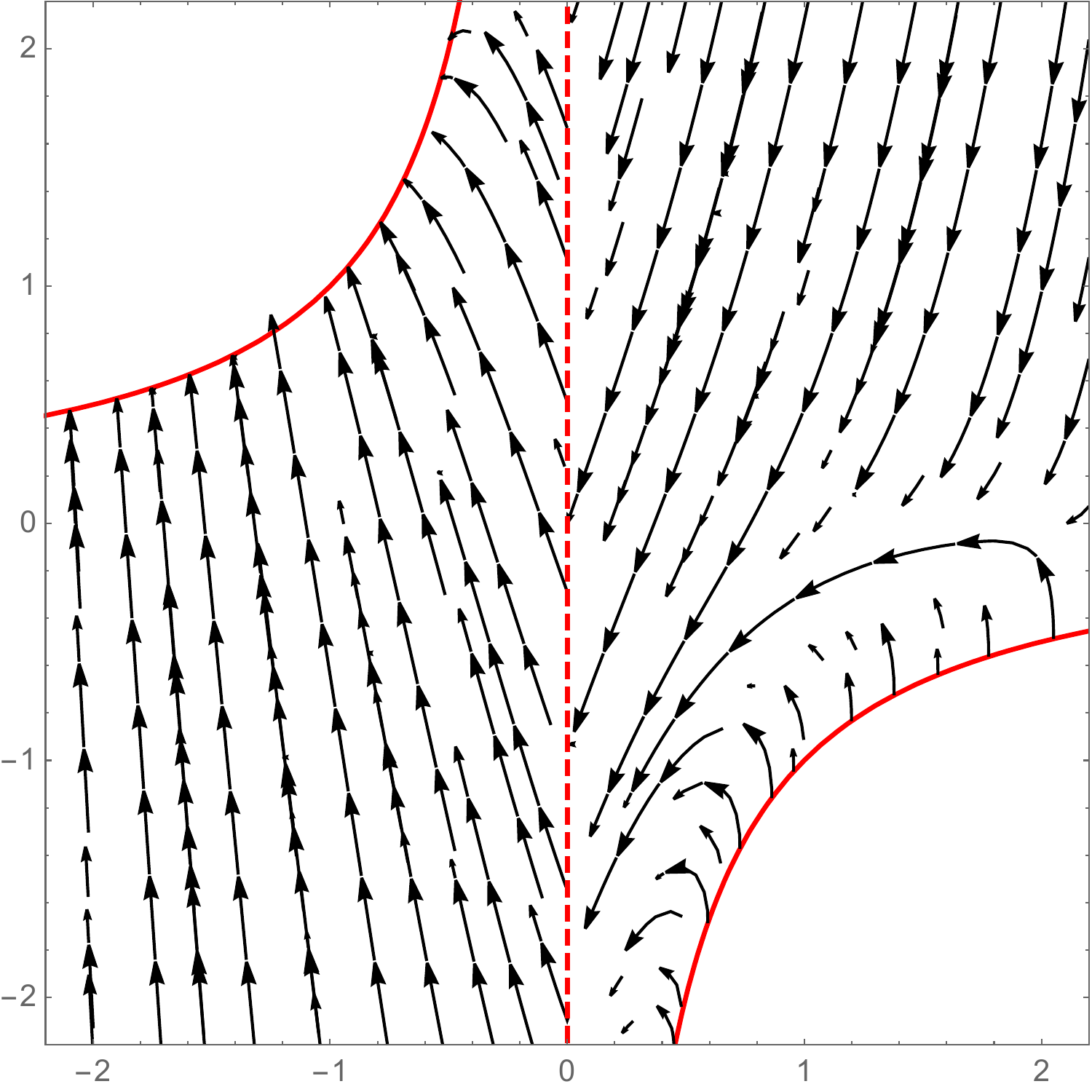}
\put(-120,124){$\widetilde \Sigma^-$}
\put(-45,31){$\widetilde \Sigma^+$}
\end{minipage}
\begin{minipage}{0.45\linewidth}
\centering
\includegraphics[width=5.5cm]{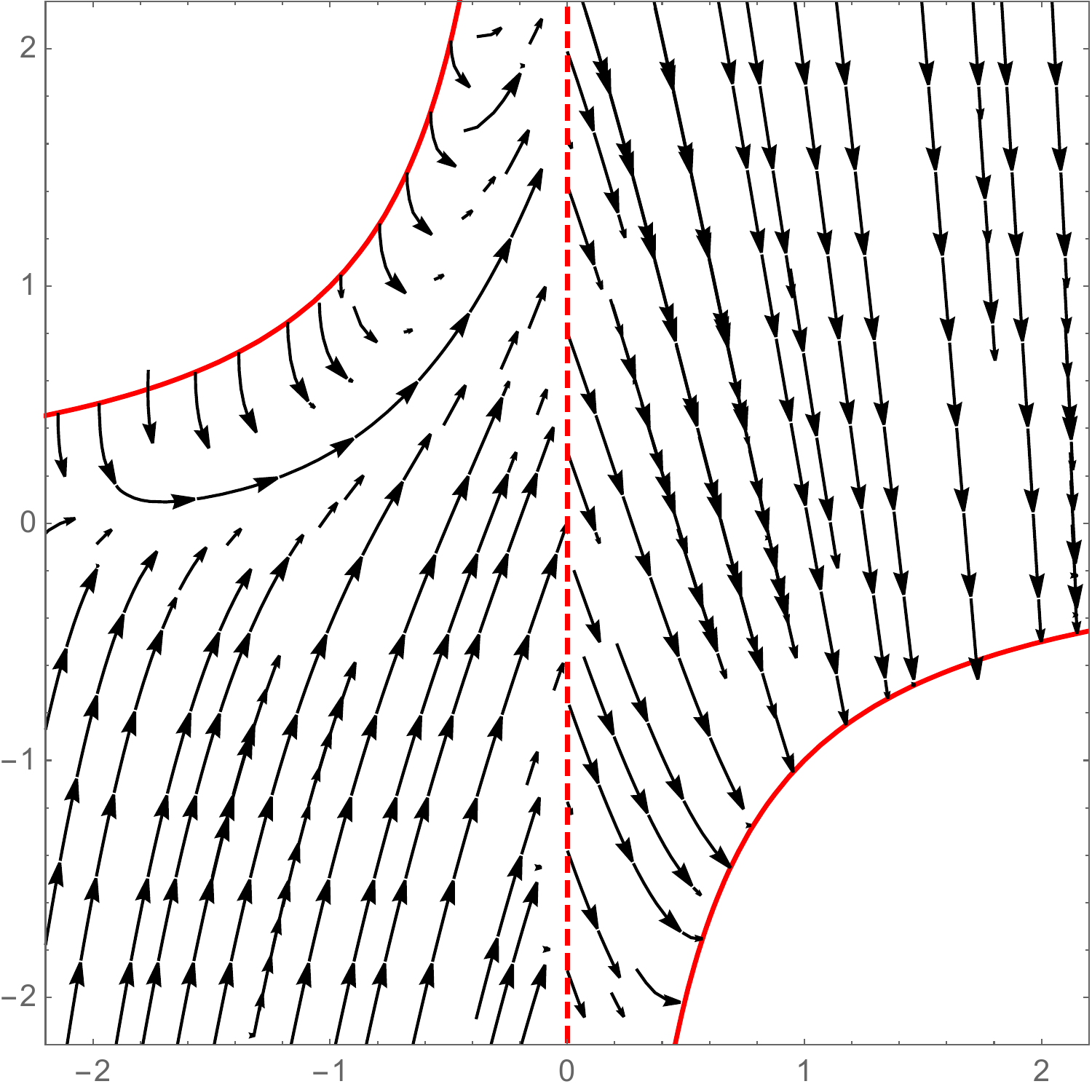}
\put(-120,124){$\widetilde \Sigma^-$}
\put(-45,31){$\widetilde \Sigma^+$}
\end{minipage}
\caption{Phase space of the vector fields (left)  $\widetilde X_{1}^-$ and (right) $\widetilde X_{1}^+$  defined for $u v\geq-1.$ On the dashed line both vector fields (left)  $\widetilde X_{1}^-$ and (right) $\widetilde X_{1}^+$ are discontinuous.}\label{f1}
\end{figure}

Consider the sections $\Sigma_D^1$ and $\Sigma_I^1$ on the hyperbola $u\,v+1=0,$ given by
\[
\begin{aligned}
\Sigma_D^1=&\{(u,v): v =-1/u: u_{i,1}^D \le u\le u_{f,1}^D\}\quad \text{and}\\
\Sigma_I^1=&\{(u,v): v =-1/u:  u_{i,1}^I \le u\le u_{f,1}^I\},
\end{aligned}
\]
where
\[
u_{i,1}^D= 1/2^{3/4}, \quad  u_{f,1}^D= 2, \quad u_{i,1}^I= -\frac{1+\sqrt{2}}{2^{3/4}} \quad \text{and} \quad  u_{f,1}^I=-1/2^{3/4} .
\]

We are going to show that the flow of  $\widetilde X_1^-$ induces a map
\begin{equation}\label{eq:P1}
P_1:\Sigma_D^1\to\Sigma_I^1.
\end{equation}
To do that, we will construct a compact connected trapping region $K^1$ such that $\Sigma_D^1\cup\Sigma_I^1 \subset \partial K^1$  and $\widetilde X_1^-$ points inward everywhere on $\partial K^1$ except at $\Sigma_I^1$ (see Figure \ref{trapping1}). Since $\widetilde X_1^-$ does not have singularities in $K^1$ and both vector fields  $\widetilde X_{1,L}^-$ and  $\widetilde X_{1,R}^-$ are transversal on $u=0$ and point to the left, applying the Poincar\'{e}-Bendixson Theorem (see, for instance, \cite[Theorem 1.25]{DLA06}) first in the region $K^1\cap\{u\geq0\}$ and then in the region $K^1\cap\{u\leq0\}$, we conclude that:
\begin{itemize}
\item[(C1)] for each point $p\in\Sigma_D^1$, there exists $t(p)>0$ such that $P_1(p):=\phi_1^-(t(p),p)\in \Sigma_I^1$.
\end{itemize}

Let $K^1$ (see Figure \ref{trapping1}) be the compact region delimited by
\[
\partial K^1=\Sigma_D^1\cup\Sigma_U^1\cup R_1 \cup \cdots \cup R_5,
\]
where
\[
\begin{aligned}
R_1=&\{(u,0), 0<u\le 5/2\},\\
R_2=&\{(u,v): v =2 u -5/2: u_{f,1}^D \le u \le 5/2\},\\
R_3=&\{(u,v): v =2 \sqrt{2} |u| -2^{7/4}: 0 \le u \le u_{i,1}^D\}, \\
R_4=& \{(u,v): v =2 \sqrt{2}  |u| -2^{7/4}: u_{i,1}^I\le  u <0\}, \\
R_5=&\{(u,v): v =2 \sqrt{2}  |u|:u_{f,1}^I \le  u \le  0\}.
\end{aligned}
\]

\begin{figure}[H]
\begin{overpic}[width=7cm]{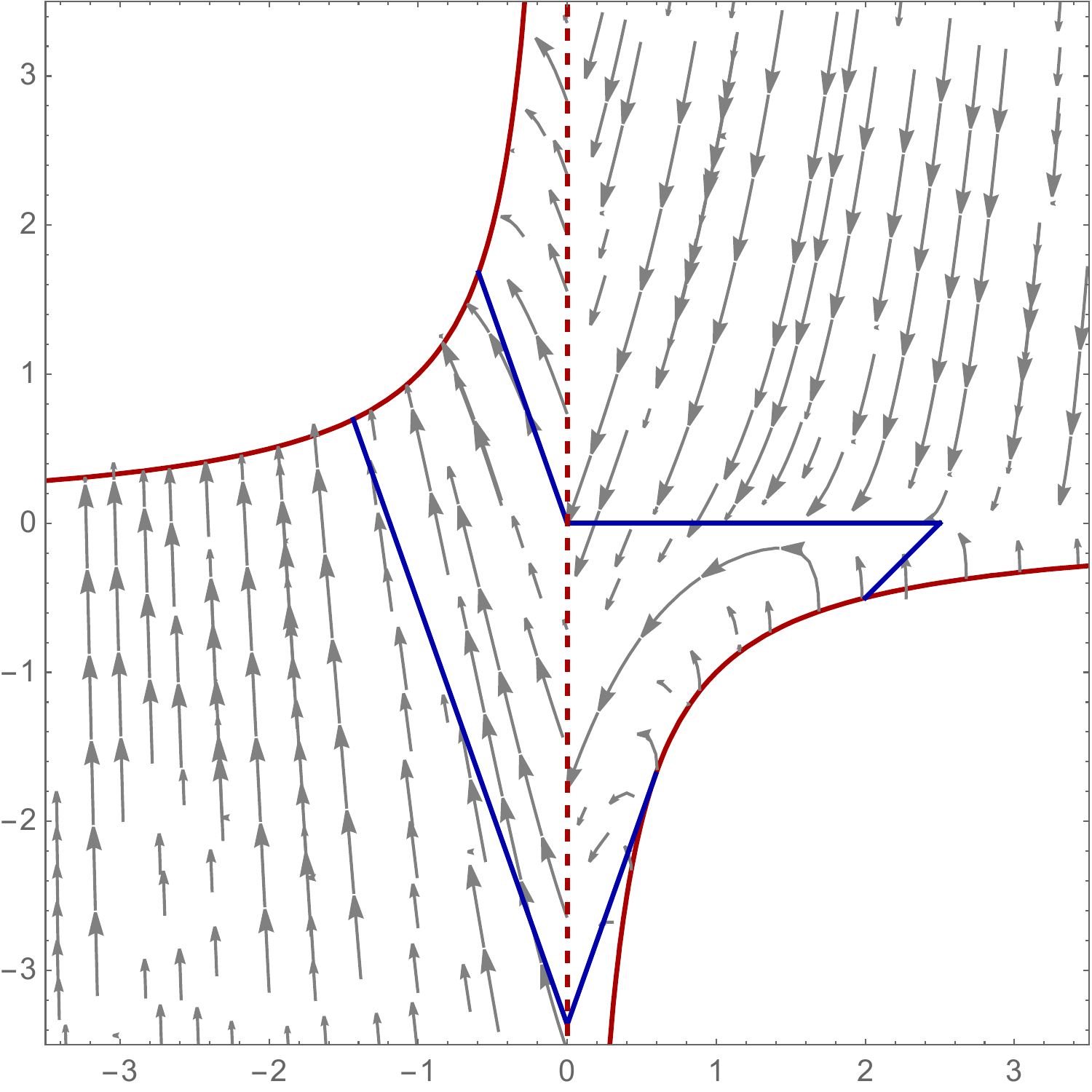}
\put(67,55){$R_1$}
\put(85,46){$R_2$}
\put(70,35){$\Sigma_D^1$}
\put(54,10){$R_3$}
\put(34,30){$R_4$}
\put(31,69){$\Sigma_I^1$}
\put(49,65){$R_5$}
\put(48,35){$K^1$}
	\end{overpic}
\caption{Phase space of the vector field  $\widetilde X_{1}^-$ defined for $u v\geq-1$ and the trapping region $K^1$.}
\label{trapping1}
\end{figure}
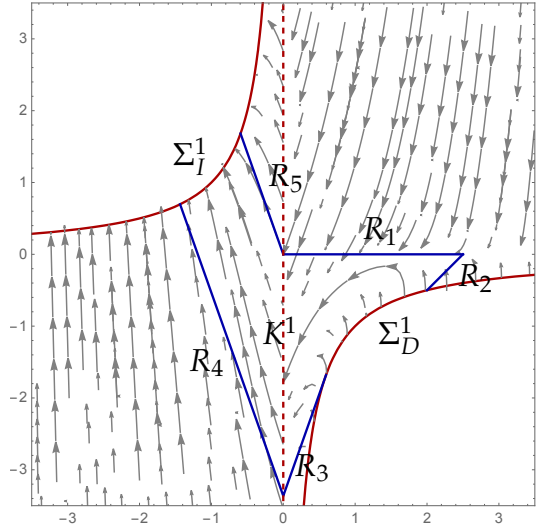

In what follows, we are going to analyze the behavior of the vector field $\widetilde X_{1}^-$ on each component of the boundary of $K^1$.

\noindent{\bf Behavior of $\widetilde X_1^{-}$ on $\Sigma_D^1$ and $\Sigma_I^1$.} The hyperbola $g(u,v)=u v + 1=0$  on $\Sigma_D^1$ satisfies
\[
\langle \nabla g(u,v),\widetilde X_1^{-}(u,v)\rangle|_{(u,v)\in \Sigma_D^1}= u,
\]
which does not vanish on $\Sigma_D^1$.  Note that the flow of the vector field $\widetilde X_1^{-}$ points inwards $K^1$ on $\Sigma_D^1$, because for instance $\widetilde X_1^-(1,-1)=(0,1)$.

The curve $g(u,v)=u v + 1=0$ on $\Sigma_I^1$ satisfies
\[
\langle \nabla g(u,v),\widetilde X_1^{-}(u,v)\rangle|_{(u,v)\in\Sigma_I^1} =u(1+8 u^4),
\]
which does not vanish on $\Sigma_I^1$. Note that the flow of the vector field $\widetilde X_1^{-}$ points outwards $K^1$ on $\Sigma_I^1$, because for instance $\widetilde X_1^-(-1,1)=(-2,7)$.

\bigskip

\noindent{\bf Behavior of $\widetilde X_1^{-}$ on  $R_1$.}
The curve $g(u,v)=v=0$ satisfies
\[
\begin{split}
\langle \nabla g(u,v),\widetilde X_1^{-}(u,v)\rangle|_{g(u,v)=0} &  = - 4  (1+u^4 -u^2 \sqrt{2 + u^4}).
\end{split}
\]
We will show that this derivative does not vanish for $u>0$. We proceed by contradiction. Assume that $1+u^4 -u^2 \sqrt{2 + u^4}=0$. Then,
\[
\sqrt{2 + u^4} = \frac{1+u ^4}{u^2} \quad \Leftrightarrow \quad 2 + u^4 = \left(\frac{1+u ^4}{u^2}\right)^2 = \frac{1+u^{8} + 2 u^4}{u^4} \quad \Leftrightarrow \quad -\frac{1}{u^4}=0,
\]
which is not possible. This shows that the flow along $R_1$ points always inwards $K^1$, because for instance $\widetilde X_1^-(1,0)= (1+\sqrt{3},-4 (2+\sqrt{3}))$.

\bigskip

\noindent{\bf Behavior of $\widetilde X_1^{-}$ on $R_2$.}
The curve $g(u,v)=v- u +5/2=0$ satisfies
\[
\langle \nabla g(u,v),\widetilde X_1^{-}(u,v)\rangle|_{g(u,v)=0} = -4 + \frac{25 }2 u  -6 u^2 - 4 u^4 + (4 u^2+1) \sqrt{2 - 5 u + 2 u^2 + u^4} .
\]
We will show that this derivative does not vanish on $R_2$ that is, for $u_{f,1}^D \le u \le 5/2$. We proceed by contradiction. Assume that it is zero at some $u$. Then, proceeding analogously to $R_1$ to get rid of the square root, we obtain
\[
\left( -4 + \frac{25 }2 u  -6 u^2 - 4 u^4 \right)^2 =  (1+4u^2)^2 (2 - 5 u + 2 u^2 + u^4),
\]
or equivalently
\[
p_0(u) := -14 + 95 u - \frac{745}4 u^2 + 110 u^3 - 19 u^4 + 20 u^5 - 8 u^6=0.
\]
Using the Sturm procedure  we conclude that the polynomial $p_0(u)$  does not vanish on $u_{f,1}^D  < u < 5/2$ which is a contradiction (see Appendix). Hence, the flow of the vector field  $\widetilde X_1^{-}$  points inwards $K^1$ on $R_2.$

\bigskip

\noindent{\bf Behavior of $\widetilde X_1^{-}$ on $R_3$ and $R_4$.} On the curve $g(u,v)=v-2 \sqrt{2} |u| -2^{7/4}=0$ we have: for $u>0$
\[
\begin{split}
\langle \nabla g(u,v),\widetilde X_1^{-}(u,v)\rangle|_{g(u,v)=0} &=p_1(u): =-4 + 10 \cdot 2^{3/4}  u-12 \sqrt{2}  u^2  - 4 u^4 \\
& \phantom{\le} + (4 u^2 + 2 \sqrt{2}) \sqrt{2 - 4 \cdot 2^{3/4} u + 4 \sqrt{2} u^2 + u^4},
\end{split}
\]
and for $u<0$,
\[
\begin{split}
\langle \nabla g(u,v),\widetilde X_1^{-}(u,v)\rangle|_{g(u,v)=0} &=p_{2}(u): = 4 - 10 \cdot 2^{3/4}  u-12 \sqrt{2}  u^2  + 4 u^4 \\
& \phantom{\le}  + (4 u^2 -  2 \sqrt{2}) \sqrt{2 - 4 \cdot 2^{3/4} u + 4 \sqrt{2} u^2 + u^4}.
\end{split}
\]
Of course on $u=0$ the above derivative is zero but since it is a quadratic tangency it is enough to show that $p_1(u)$ does not vanish  on $0 < u \le u_{i,1}^D$ and that the $p_{2}(u)$ does not vanish on $u_{i,1}^I \le u  <0$.

We proceed by contradiction. Assume first that $p_1(u)$ has a zero. Then, proceeding as in $R_1$ to get rid of the square root, we have that $p_1(u)=0$
implies
\[
p_3(u)=12 \cdot 2^{3/4} - 58 \sqrt{2} u + 88 \cdot 2^{1/4} u^2 - 38 u^3 + 4 \cdot 2^{3/4} u^4 - 4 \sqrt{2}  u^5=0.
\]
Using the Sturm procedure  we have that the polynomial $p_3(u)$  has a unique simple real zero on $u>0$ (see Appendix). Moreover, we have that
\[
p_1(0)=0, \quad   p_1'(0)=6 \cdot 2^{3/4}, \quad  p_1(u_{i,1}^D) = 1.
\]
These computations show that if $p_1(u)$ has a zero in $(0,u_{i,1}^D)$ then it has another (or a multiple) zero on that interval. This would produce another zero of $p_3(u)$ in $(0, \infty)$, which is a contradiction. Hence the flow of the vector field  $\widetilde X_1^{-}$ points inwards $K^1$ on $R_3$.

On the other hand assume that $p_2(u)$ has a zero. Then, proceeding as in $R_1$ to get rid of the square root, we have that $p_2(u)=0$ implies
\[
p_4(u)=12 \cdot 2^{3/4} - 42 \sqrt{2} u - 88 \cdot 2^{1/4} u^2 - 38 u^3 + 4 \cdot 2^{3/4} u^4 + 4 \sqrt{2}  u^5=0.
\]
Using the Sturm procedure  we conclude that the polynomial $p_4(u)$  has a unique simple real zero on $u<0$ (see Appendix). Moreover, we have that
\[
p_2(0)=0, \quad   p_2'(0)=-6 \cdot 2^{3/4}, \quad    p_2(u_{i,1}^I)= 38.5802...
\]
These computations show that if $p_2(u)$ has a zero in $(u_{i,1}^I,0)$, then either it is multiple or $p_2(u)$ has at least two negative zeros in contradiction with uniqueness of negative zeros provided by Sturm procedure.
Hence the flow of the vector field  $\widetilde X_1^{-}$ points inwards $K^1$ on  $R_4$.

\bigskip

\noindent{\bf Behavior of $\widetilde X^{-}$ on $R_5$.}
Note that on the curve $g(u,v)=v- 2\sqrt{2} |u| =0$ we have
\[
\langle \nabla g(u,v),\widetilde X_1^{-}(u,v)\rangle|_{g(u,v)=0} =p_5(u):= 4 - 12 \sqrt{2} u^2 + 4u^4 +(4 u^2 - 2 \sqrt{2})\sqrt{2 - 4 \sqrt{2} u^2 + u^4}.
\]
Of course on $u=0$ the above derivative is zero but since it is a quadratic tangency it is sufficient to show that this derivative does not vanish on $u_{f,1}^I < u < 0$.  We proceed by contradiction. Assume that it is zero. Then, proceeding as in the curve $R_1$ to get rid of the square root, we have that $p_5(u)=0$ implies
\[
\frac{2 u^2 (4 \sqrt{2} - 19 u^2 + 2 \sqrt{2}  u^4)}{ (\sqrt{2} - 2 u^2)^2}=0.
\]
Note that the denominator does not vanish on $ u_{f,1}^I < u < 0$, because the unique negative real solution of $\sqrt{2} - 2 u^2=0$, which is $-1/2^{1/4},$ is outside the mentioned interval. On the other hand the unique two negative real solutions of the numerator are
\[
u_1=-\frac 1 2  \left(\frac{19 \sqrt{2} - 3 \sqrt{66}}2\right)^{1/2}\quad \text{and} \quad  u_2=-\frac 1 2  \left(\frac {19 \sqrt{2} + 3 \sqrt{66}}2\right)^{1/2}.
\]
We observe that $u_1$ belongs to the interval $ u_{f,1}^I < u < 0$ while $u_2$ does not. However, $u_1$ is not a solution of $p_5(u)=0$ because
\[
p_5(u_1) = -\frac 3 8  \left(-78 +14 \sqrt{33} + 3 \sqrt{22 (19-3 \sqrt{33})} - 5 \sqrt{6 (19-3 \sqrt{33})} \right) \ne 0.
\]
This contradiction shows that the flow of the vector field  $\widetilde X_1^{-}$ points inwards $K^1$ on $R_5$.

\bigskip

\noindent{\bf Existence of $p_1^*$.} Consider the map $h_1 \colon \Sigma_D^1 \to \R$ defined by $h_1(p)=\pi (P_1(p)+p)$, where $\pi:\R^2\to\R$ is the projection onto the first coordinate (see (C1)). Note that by the continuous dependence of the flow of $\widetilde X_1^-$ with respect to the initial conditions the map $h_1$ is continuous. Moreover, the image of the point $(u_{f,1}^D,-1/u_{f,1}^D)$ by $P_1$ (which is the point $w_4$ in Figure \ref{fig.3q1}) is inside $\Sigma_I^1$, and  its symmetric is below its image because $u_{f,1}^D + u_{i,1}^I>0$. Hence $h_1(u_{f,1}^D,-1/u_{f,1}^D)>0$. On the other hand, the image of the point $(u_{i,1}^D,-1/u_{i,1}^D)$ by $P_1$ (which is the point $w_3$ in Figure \ref{fig.3q1}) is inside $\Sigma_I^1$, and  its symmetric is above its image because  $u_{i,1}^D +u_{f,1}^I<0$. Therefore $h_1(u_{i,1}^D,-1/u_{i,1}^D)<0$. Thus, by continuity, there exists $p_1^*\in\Sigma_D^1$ such that $h_1(p_1^*)=0$, i.e, $P_1(p_1^*)=-p_1^*$ and so $\phi_1^-(t^*_1,p_1^*)=-p_1^*$ as we wanted to prove.

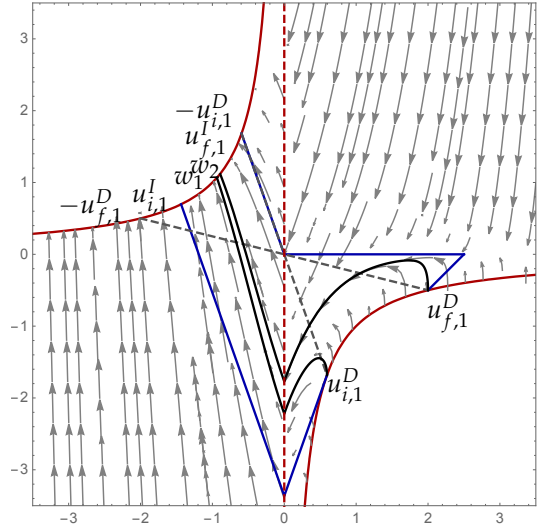
\begin{figure}[H]
\begin{overpic}[width=7cm]{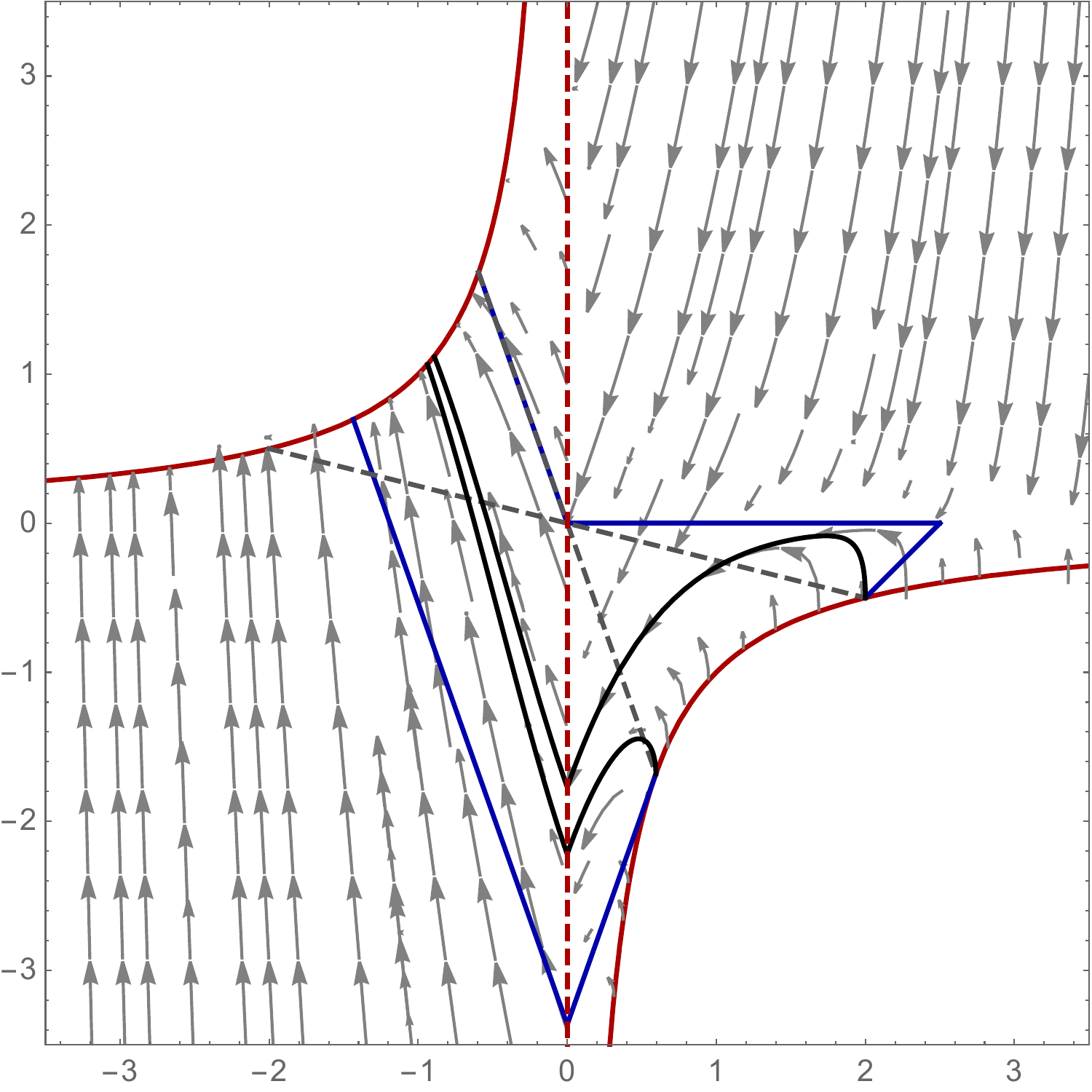}
\put(34,67.5){\footnotesize{$w_2$}}
\put(31,65.5){\footnotesize{$w_1$}}
\put(9,60){\footnotesize{$-u_{f,1}^D$}}
\put(31,78){\footnotesize{$-u_{i,1}^D$}}
\put(33,73){\footnotesize{$u_{f,1}^I$}}
\put(23,62){\footnotesize{$u_{i,1}^I$}}
\put(79,40){\footnotesize{$u_{f,1}^D$}}
\put(60,26){\footnotesize{$u_{i,1}^D$}}
	\end{overpic}
\caption{Idea of the obtention of the fixed point: $w_1 =P(u_{i,1}^D)$ and $w_2=P(u_{f,1}^D)$.}\label{fig.3q1}
\end{figure}

\bigskip

\subsection{ Proof of Theorem \ref{main} for the polynomial case $q=2$.} \label{sec:q2}

In this case, it follows from \eqref{eq:vector} that
\[
\widetilde X_2^{\pm}:
\left\{
\begin{aligned}
u'=&u^3 \pm\sqrt{u^{6}+2uv+2},\\
v'=&u\left(-u v-3 \left(u^3 \pm\sqrt{u^{6}+2uv+2}\right)^2\right),
\end{aligned}
\right.
\]
for $u v+1\geq0$. See Figure \ref{fig2} for the phase space of the vector field $\widetilde X_{2}^-$ (in the left) and  $\widetilde X_{2}^+$ (in the right).

\begin{figure}[H]
\begin{minipage}{0.45\linewidth}
\centering
\includegraphics[width=5.5cm]{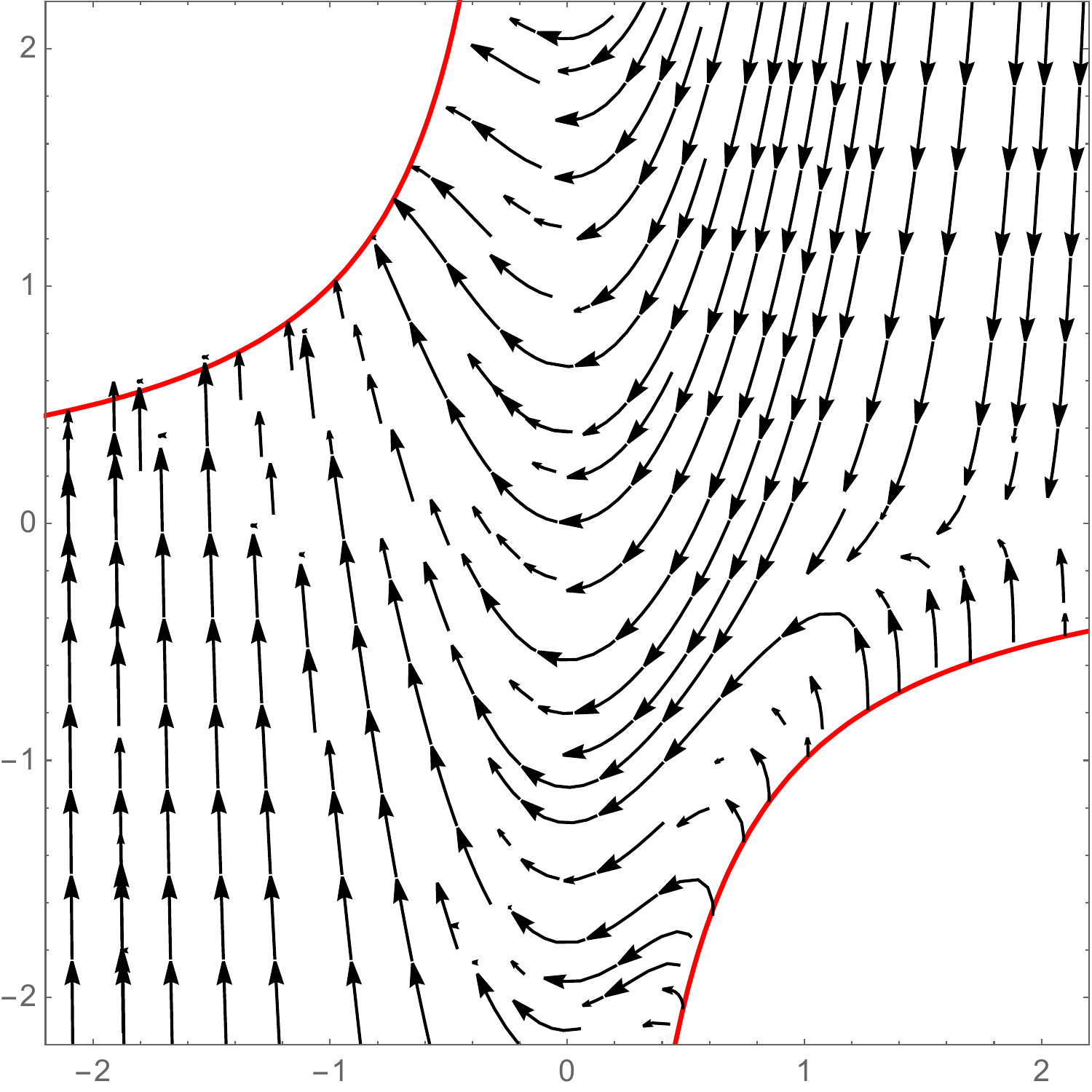}
\put(-120,124){$\widetilde \Sigma^-$}
\put(-45,31){$\widetilde \Sigma^+$}
\end{minipage}
\begin{minipage}{0.45\linewidth}
\centering
\includegraphics[width=5.5cm]{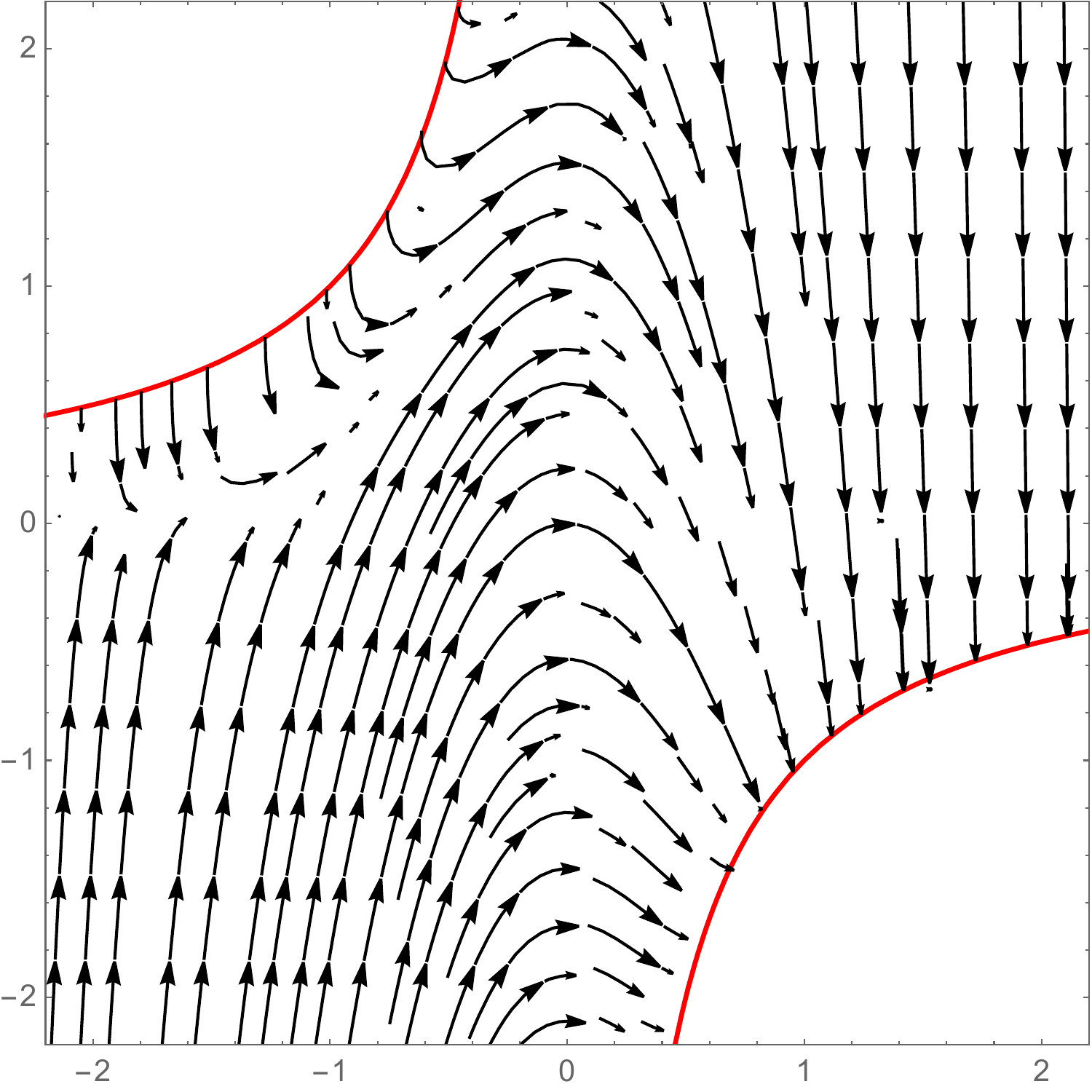}
\put(-120,124){$\widetilde \Sigma^-$}
\put(-45,31){$\widetilde \Sigma^+$}
\end{minipage}
\caption{Phase space of the vector fields (left)  $\widetilde X_{2}^-$ and (right) $\widetilde X_{2}^+$  defined for $u v\geq-1.$}\label{fig2}
\end{figure}

Consider the arcs $\Sigma_D^2$ and $\Sigma_I^2$ on the hyperbola $u\,v+1=0,$ given by
\[
\begin{aligned}
\Sigma_I^2=&\{(u,v): v =-1/u:  u_{i,2}^I \le u\le u_{f,2}^I\},
\end{aligned}
\]
where
\[
u_{i,2}^D= 2^{-1/6} 3^{-1/3}, \quad  u_{f,2}^D= 2, \quad u_{i,2}^I=  -2^{5/6}/3^{1/3} \quad \text{and} \quad  u_{f,2}^I= -2^{1/6}/3^{1/3}.
\]

As for the discontinuous case $q=1$ we are going to show that in this case the flow of  $\widetilde X_2^-$ induces a map
\begin{equation}\label{eq:P2}
P_2:\Sigma_D^2\to\Sigma_I^2,
\end{equation}
by constructing a trapping region $K^2$ such that $\Sigma_D^2\cup\Sigma_I^2\subset \partial K^2$ and $\widetilde X_2^-$ will point inwards $K^2$ everywhere on $\partial K^2$ except at $\Sigma_I^2$, see Figure \ref{trapping2}. Since $\widetilde X_2^-$ does not have singularities in $K^2$, from the Poincar\'{e}-Bendixson Theorem (see, for instance, \cite[Theorem 1.25]{DLA06}), we conclude that:
\begin{itemize}
\item[(C2)] for each point $p\in\Sigma_D^2$, there is $t(p)>0$ such that $P_2(p):=\phi_2^-(t(p),p)\in \Sigma_I^2$.
\end{itemize}

Let $K^2$ be the compact region delimited by
\[
\partial K^2=\Sigma_D^2\cup\Sigma_U^2\cup S_1 \cup \cdots \cup S_5,
\]
where
\[
\begin{aligned}
S_1=&\{(u,0), 0<u\le 9/4\},\\
S_2=&\{(u,v): v =2 u -9/2: u_{f,2}^D \le u \le 9/4\},\\
S_3=&\{(u,v): v =3 u^2/\sqrt{2} -3^{4/3}/2^{5/6}: 0 \le u \le u_{i,2}^D\}, \\
S_4=& \{(u,v): v =3 u^2/\sqrt{2} -3^{4/3}/2^{5/6}: u_{i,2}^I\le  u <0\}, \\
S_5=&\{(u,v): v =3 u^2/\sqrt{2}:u_{f,2}^I \le  u \le  0\}.
\end{aligned}
\]

\begin{figure}[H]
\begin{overpic}[width=7cm]{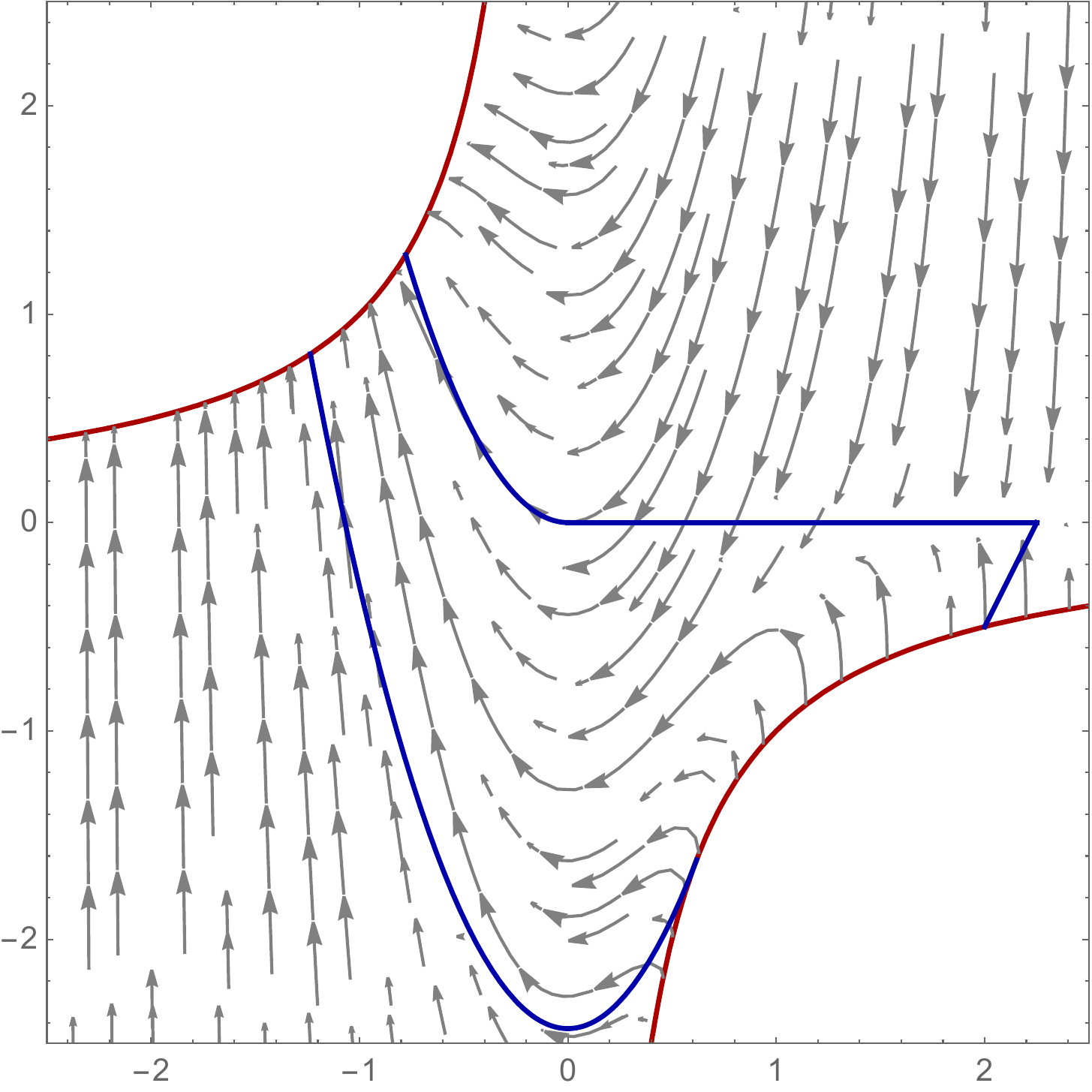}
\put(75,55){$S_1$}
\put(94,46){$S_2$}
\put(70,28){$\Sigma_D^2$}
\put(58,7){$S_3$}
\put(30,27){$S_4$}
\put(27,73){$\Sigma_I^2$}
\put(44,65){$S_5$}
\put(48,35){$K^2$}
	\end{overpic}
\caption{Phase space of the vector field  $\widetilde X_{2}^-$ defined for $u v\geq-1$ and the trapping region $K^2$.}\label{trapping2}
\end{figure}
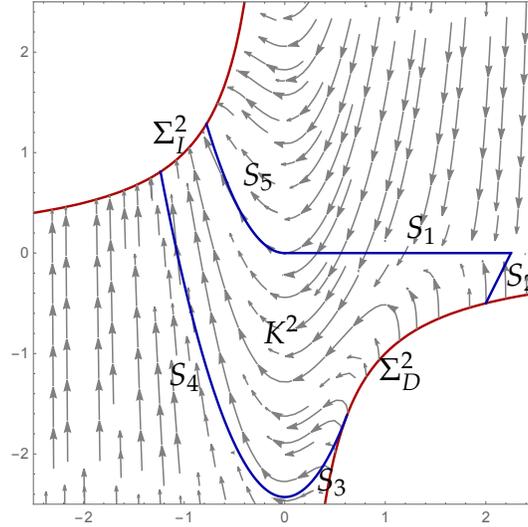

In what follows, we are going to analyze the behavior of the vector field $\widetilde X_{2}^-$ on each component of the boundary of $K^2$.

\noindent{\bf Behavior of $\widetilde X_2^{-}$ on $\Sigma_D^2$ and $\Sigma_I^2$.}
The hyperbola $g(u,v)=u v + 1=0$  satisfies
\[
\langle\nabla g(u,v),\widetilde X_2^{-}\rangle |_{g(u,v)=0} = u^2,
\]
which does not vanish on $\Sigma_D^2\cup \Sigma_I^2$.  Note that the flow of the vector field $\widetilde X_2^{-}$ points inwards $K^2$ on $\Sigma_D^2\cup \Sigma_I^2$, because, for instance, $\widetilde X_2 ^-(1,-1)=(0,1)$ and $\widetilde X_2^-(-1,1)=(-2,11)$.

\bigskip

\noindent{\bf Behavior of $\widetilde X_2^{-}$ on $S_1$.} The curve $g(u,v)=v=0$  satisfies
\[
\begin{split}
\langle\nabla g(u,v),\widetilde X_2^{-}\rangle|_{g(u,v)=0} &  = 6 u (1+u^6 -u^3 \sqrt{2 + u^6}).
\end{split}
\]
We will show that this derivative does not vanish on $u>0$. We proceed by contradiction. Assume that $6 u (1+u^6 -u^3 \sqrt{2 + u^6})=0$. Then, since $u > 0$ we have $1+u^6 -u^3 \sqrt{2 + u^6}=0$, and so
\[
\sqrt{2 + u^6} = \frac{1+u ^6}{u^3} \quad \Leftrightarrow \quad 2 + u^6 = \left(\frac{1+u ^6}{u^3}\right)^2 = \frac{1+u^{12} + 2 u^6}{u^6} \quad \Leftrightarrow \quad -\frac{1}{u^6}=0,
\]
which is not possible. This shows that the flow of the vector field  $\widetilde X_2^{-}$ points inwards $K^2$ on $S_1$, because for instance $\widetilde X_2^-(1,0)=(-1-\sqrt{2},-3 (1+\sqrt{3}))$.

\bigskip

\noindent{\bf Behavior of $\widetilde X_2^{-}$ on $S_2$.}
The curve $g(u,v)=v-2 u +9/2=0$  satisfies
\[
\langle\nabla g(u,v),\widetilde X_2^{-}\rangle|_{g(u,v)=0} = -6 u + \frac{63}2u^2 - 16 u^3 - 6 u^7 + (2+6u^4)\sqrt{u^6 + (u-2) (4u-1 )}.
\]
We will show that this derivative does not vanish on $S_2$, i.e, for $u_{f,2}^D \le u \le 9/4$. We proceed by contradiction. Assume that it is zero at some $u$. Then, proceeding as in $S_1$ to get rid of the square root, we obtain
\[
( 6 u - \frac{63}2u^2 + 16 u^3 + 6 u^7)^2 =  (2+6u^4)^2 (u^6 + (-2 + u) (-1 + 4 u)),
\]
and so $u^2 p_6(u)=0$, where
\[
p_6(u) := 32  - 144 u - 80 u^2 + 1512 u^3 - 4545 u^4 + 3168 u^5 - 624 u^6 + 216 u^9 - 96 u^{10}.
\]
Using the Sturm procedure we obtain that the polynomial $p_6(u)$  does not vanish on $u_{f,2}^D  < u < 9/4$ which is a contradiction (see Appendix). Hence, the flow of the vector field  $\widetilde X_2^{-}$ points inwards $K^2$ on $S_2.$

\bigskip

\noindent{\bf Behavior of $\widetilde X_2^{-}$ on $S_3$ and $S_4$.}
Note that on $g(u,v)=v-3 u^2/\sqrt{2} -3^{4/3}/2^{5/6}=0$ we have
\[
\begin{split}
\langle\nabla g(u,v),\widetilde X_2^{-}\rangle|_{g(u,v)=0} & =-u \bigg( 6 - \frac{21 \cdot 3^{1/3} u}{2^{5/6}} + \frac{27 u^3}{\sqrt{2}} + 6 u^6 + (3 \sqrt{2}+ 6 u^3) \\
& \phantom{\le} \times \sqrt{2 - 3 \cdot 2^{1/6}  3^{1/3} u + 3 \sqrt{2} u^3 + u^6}\bigg).
\end{split}
\]
Of course on $u=0$ the above derivative is zero but since it is a quadratic tangency it is enough to show that this derivative does not vanish neither on $0 < u \le u_{i,2}^D$, nor on $u_{i,2}^I \le u  <0$. We proceed by contradiction. Assume that
\[
p_7(u):=6 - \frac{21 \cdot 3^{1/3} u}{2^{5/6}} + \frac{27 u^3}{\sqrt{2}} + 6 u^6 + (3 \sqrt{2}+ 6 u^3) \sqrt{2 - 3 \cdot 2^{1/6}  3^{1/3} u + 3 \sqrt{2} u^3 + u^6}
\]
has a zero. Then, proceeding as in $S_1$ to get rid of the square root, we have that $p_7(u)=0$ implies
\[
\begin{split}
p_8(u)& := 32 \cdot 2^{1/6}  3^{1/3} - 49 \cdot 2^{1/3} 3^{2/3} u - 16 \sqrt{2} u^2 + 78\cdot  2^{2/3}  3^{1/3} u^3 - 58 u^5 \\
&+ 8 \cdot 2^{1/6}  3^{1/3} u^6  - 8 \sqrt{2} u^8=0.
\end{split}
\]
Using the Sturm procedure  we show that the polynomial $p_8(u)$  has a unique simple real zero on $u>0$, and a unique simple real zero on $u<0$ (see Appendix). Moreover we have that
\[
p_7(0)=0, \,   p_7'(0)=0, \,  p''_7(0)=12 \cdot 2^{1/6} 3^{1/3}, \, p_7(u_{i,2}^D) = 0.617715, \, p_7(u_{i,2}^I)= 31.7094.
\]
These computations show that if $p_7(u)$ has a zero in $(0,u_{i,2}^D)$, then it has another (or a multiple) zero on that interval. This would produce another zero of $p_8(u)$ in $(0, \infty)$, which is a contradiction. On the other hand, if $p_7(u)$ has a zero in $(u_{i,2}^I,0)$ then it has another (or a multiple) zero on that interval and this would produce another zero of $p_8(u)$ in $(-\infty, 0)$, which is again a contradiction.
In short we have proved that the flow of the vector field  $\widetilde X_2^{-}$ points inwards $K^2$ on $S_3$ and $S_4$.

\bigskip

\noindent{\bf Behavior of $\widetilde X_2^{-}$ on $S_5$.}
The curve $g(u,v)=v- 3 u^2/\sqrt{2} =0$ satisfies
\[
\langle\nabla g(u,v),\widetilde X_2^{-}\rangle|_{g(u,v)=0} = p_9(u):=-\frac 3 2  u \Big(4 + 9 \sqrt{2} u^3 + 4 u^6 + (2\sqrt{2} + 4 u^3) \sqrt{2 + 3 \sqrt{2} u^3 + u^6}\Big).
\]
Of course on $u=0$ the above derivative is zero but since it is a quadratic tangency it is sufficient to show that this derivative does not vanish on $u_{f,2}^I < u < 0$. Proceeding as in $S_1$ to get rid of the square root, we have that $p_9=0$ implies
\[
\sqrt{2 + 3 \sqrt{2} u^3 + u^6} = \frac{4 + 9 \sqrt{2} u^3 + 4 u^6}{ 2\sqrt{2} + 4 u^3} \quad \Rightarrow \quad 2 + 3 \sqrt{2} u^3 + u^6 = \left(\frac{4 + 9 \sqrt{2} u^3 + 4 u^6}{ 2\sqrt{2} + 4 u^3}\right)^2,
\]
and so
\[
-\frac{u^3 (8 \sqrt{2} + 29 u^3 + 4 \sqrt{2} u^6)}{2 (\sqrt{2} + 2 u^3)^2}=0.
\]
Note that the denominator does not vanish on $ u_{f,2}^I < u < 0$, because the unique negative real solution of $\sqrt{2} + 2 u^3=0$, which is $-1/2^{1/6}$ is outside the mentioned interval. On the other hand the unique two real solutions of the numerator are
\[
u_1=-\frac 1 2  \left(\frac{29 \sqrt{2} - 3 \sqrt{130}}2\right)^{1/3}\quad \text{and} \quad  u_2=-\frac 1 2  \left(\frac {29 \sqrt{2} + 3 \sqrt{130}}2\right)^{1/3}.
\]
We observe that $u_1$ belongs to the interval $ u_{f,2}^I < u < 0$ while $u_2$ does not. However, $u_1$ is not a solution of $p_9(u)=0$ because
\[
p_9(u_1) = -\frac 1 2  \left(\frac{29 \sqrt{2} - 3 \sqrt{130}}2\right)^{1/3} \ne 0.
\]
This contradiction shows that the flow $\widetilde X_2^-$ point inwards $K^2$ on   $S_5$.

\bigskip

\noindent{\bf Existence of $p_2^*$.}
Consider the map $h_2 \colon \Sigma_D^2 \to \R$ defined by $h_2(p)=\pi (P_2(p)+p)$, where $\pi:\R^2\to\R$ is the projection onto the first coordinate (see (C2)). Note that by the continuous dependence of the flow of $\widetilde X_2^-$ with respect to the initial conditions the map $h_2$ is continuous. Moreover, the image of the point $(u_{f,2}^D,-1/u_{f,2}^D)$ by $P_2$ (which is the point $w_4$ in Figure \ref{fig.3q2}) is inside $\Sigma_I^2$, and  its symmetric is below its image because $u_{f,2}^D + u_{i,2}^I>0$. Hence $h_2(u_{f,2}^D,-1/u_{f,2}^D)>0$. On the other hand, the image of the point $(u_{i,2}^D,-1/u_{i,2}^D)$ by $P_2$ (which is the point $w_3$ in Figure \ref{fig.3q2}) is inside $\Sigma_I^2$, and  its symmetric is above its image because  $u_{i,2}^D +u_{f,2}^I<0$. Therefore $h_2(u_{i,2}^D,-1/u_{i,2}^D)<0$. Thus, by continuity, there exists $p_2^*\in\Sigma_D^2$ such that $h_2(p_2^*)=0$, i.e, $P_2(p_2^*)=-p_2^*$ and so $\phi_2^-(t^*_2,p_2^*)=-p_2^*$ as we wanted to prove.

\begin{figure}[H]
\begin{overpic}[width=7cm]{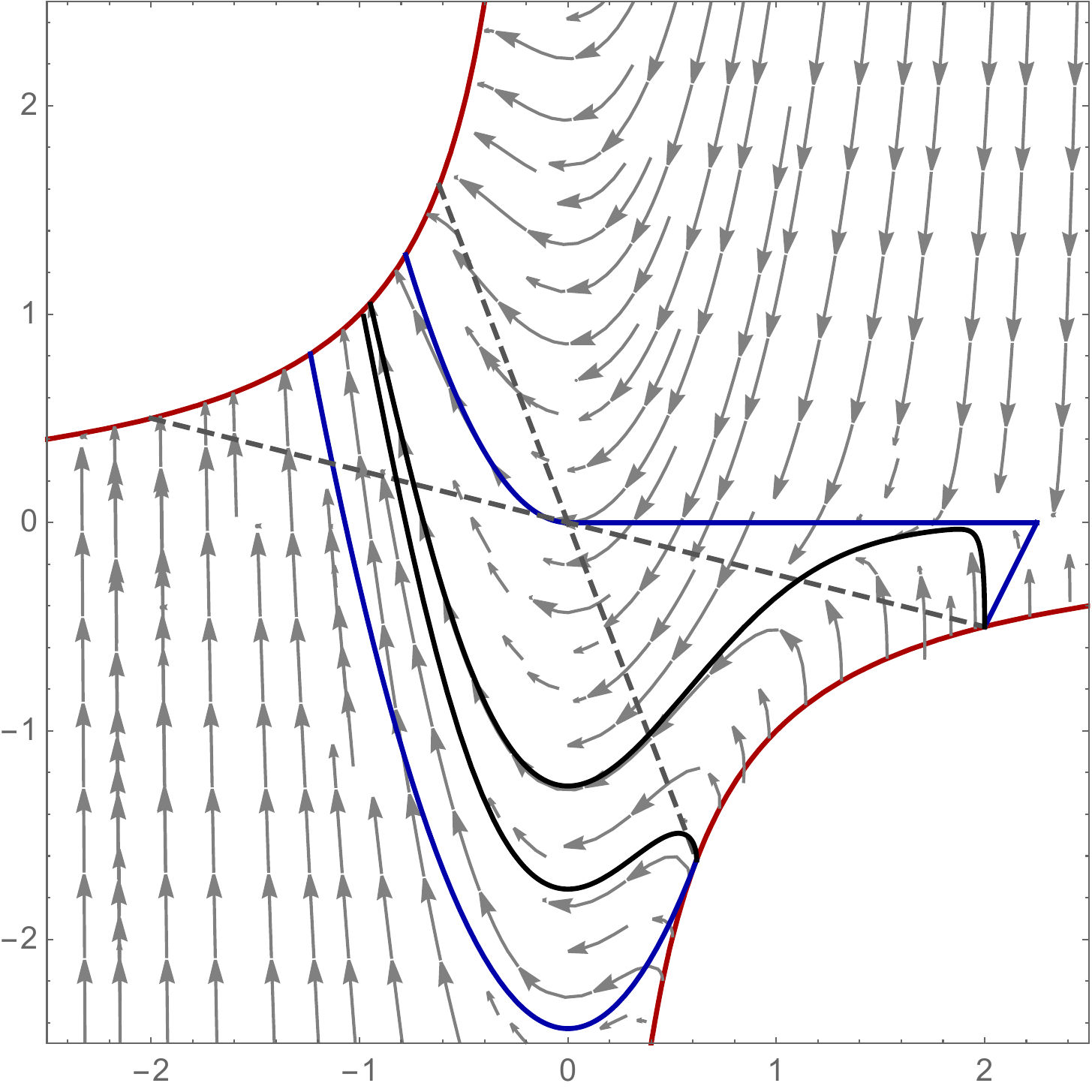}
\put(29,74.5){\footnotesize{$w_4$}}
\put(28,72.5){\footnotesize{$w_3$}}
\put(9,65){\footnotesize{$-u_{f,2}^D$}}
\put(31,85){\footnotesize{$-u_{i,2}^D$}}
\put(30,79){\footnotesize{$u_{f,2}^I$}}
\put(23,70){\footnotesize{$u_{i,2}^I$}}
\put(90,39){\footnotesize{$u_{f,2}^D$}}
\put(65,20){\footnotesize{$u_{i,2}^D$}}
	\end{overpic}
\caption{Idea of the obtention of the fixed point: $w_3 =P_2(u_{i,2}^D)$ and $w_4=P_2(u_{f,2}^D)$.}\label{fig.3q2}
\end{figure}
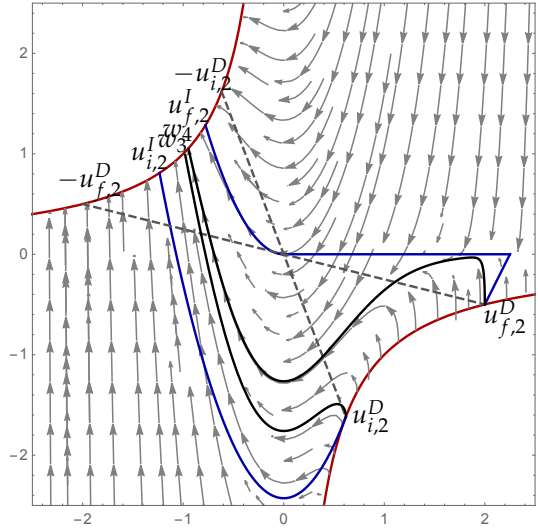

\bigskip

\subsection{ Proof of Theorem \ref{main} for the continuous case $q=3$}\label{sec:q3}

In this case, it follows from \eqref{eq:vector} that
\[
\widetilde X_2^{\pm}:
\left\{
\begin{aligned}
u'=&u \lvert u \rvert^3 \pm\sqrt{u^{8}+2uv+2},\\
v'=&\frac{\lvert u \rvert^3}{u} \left(-u v-4 \left(u \lvert u \rvert^3 \pm\sqrt{u^{8}+2uv+2}\right)^2\right).
\end{aligned}\right.
\]
for $u v+1\geq0$. See Figure \ref{fig3} for the phase spaces of the vector field $\widetilde X_{3}^-$ (in the left) and  $\widetilde X_{3}^+$ (in the right), defined for $u v\leq-1.$

\begin{figure}[H]\begin{minipage}{0.45\linewidth}
\centering
\includegraphics[width=5.5cm]{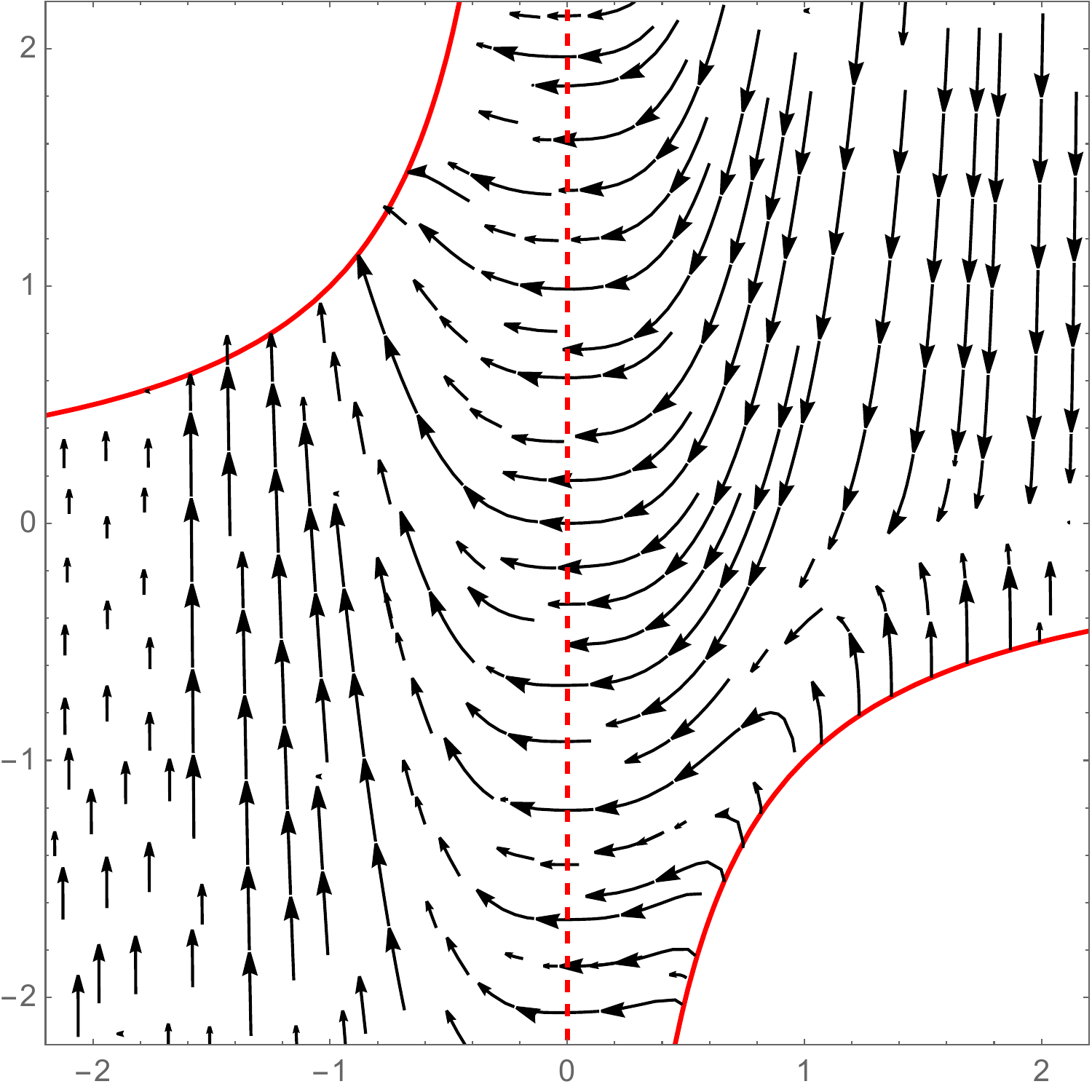}
\put(-120,124){$\widetilde \Sigma^-$}
\put(-45,31){$\widetilde \Sigma^+$}
\end{minipage}
\begin{minipage}{0.45\linewidth}
\centering
\includegraphics[width=5.5cm]{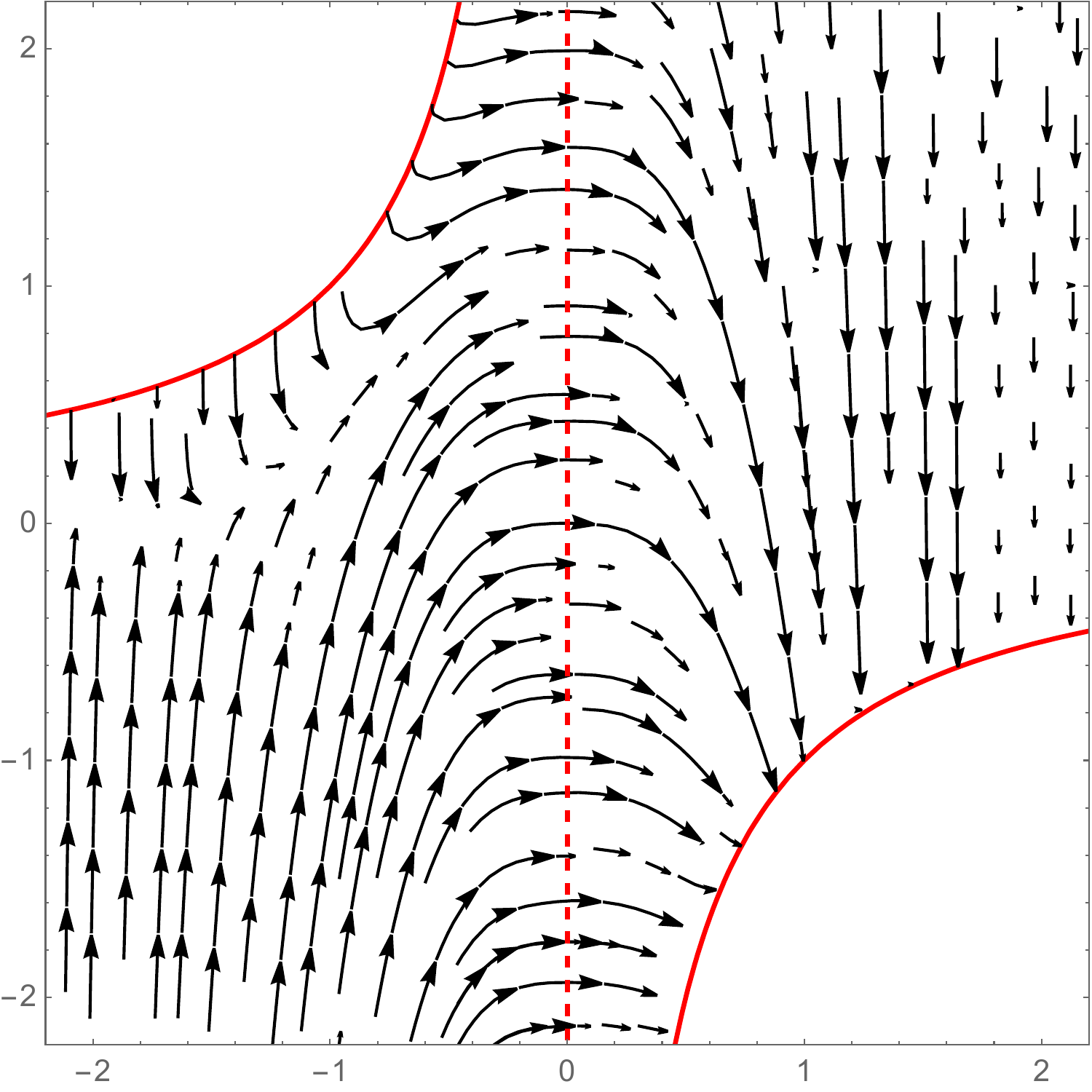}
\put(-120,124){$\widetilde \Sigma^-$}
\put(-45,31){$\widetilde \Sigma^+$}
\end{minipage}
\caption{Phase space of the vector fields (left)  $\widetilde X_{3}^-$ and (right) $\widetilde X_{3}^+$  defined for $u v\geq-1.$ On the dashed line both vector fields (left)  $\widetilde X_{3}^-$ and (right) $\widetilde X_{3}^+$ are discontinuous.}\label{fig3}
\end{figure}

Consider the sections $\Sigma_D^3$ and $\Sigma_I^3$ on the hyperbola $u\,v+1=0,$ given by
\[
\begin{aligned}
\Sigma_D^3=&\{(u,v): v =-1/u: u_{i,3}^D \le u\le u_{f,3}^D\}\quad \text{and}\\
\Sigma_I^3=&\{(u,v): v =-1/u:  u_{i,3}^I \le u\le u_{f,3}^I\},
\end{aligned}
\]
where
\[
u_{i,3}^D=\frac{1}{2^{5/8}} , \quad  u_{f,3}^D= 2, \quad u_{i,3}^I=\frac{N_{i,3}^I}{D_{i,3}^I}
\quad \text{and} \quad  u_{f,3}^I=-\frac{3^{1/4}}{2^{5/8}}
\]
being
\[
\begin{split}
D_{i,3}^I &= 2^{13/8}(1+\sqrt{2})^{1/3} \sqrt{1+\sqrt{2}-(1+\sqrt{2})^{1/3}}, \\
N_{i,3}^I &= -2-\sqrt{2}+\sqrt{2} (1+\sqrt{2})^{1/3}-\Big(-6-4 \sqrt{2}-2(1+\sqrt{2})^{2/3}+4 (1+\sqrt{2})^{4/3} \\
& \phantom{\le} + 4 (2+\sqrt{2})\sqrt{1+\sqrt{2}-(1+\sqrt{2})^{1/3}}\Big)^{1/2}.
\end{split}
\]

As for the discontinuous case $q=1$ we are going to show that in this case the flow of  $\widetilde X_3^-$ induces a map
\begin{equation}\label{eq:P3}
P_3:\Sigma_D^2\to\Sigma_I^3,
\end{equation}
by constructing a trapping region $K^3$ such that $\Sigma_D^3\cup\Sigma_I^3\subset \partial K^3$ and $\widetilde X_3^-$ will point inwards $K^3$ everywhere on $\partial K^3$ except at $\Sigma_I^3$, see Figure \ref{trapping3}. Since $\widetilde X_3^-$ does not have singularities in $K^3$, from the Poincar\'{e}-Bendixson Theorem (see, for instance, \cite[Theorem 1.25]{DLA06}),
\begin{itemize}
\item[(C3)] for each point $p\in\Sigma_D^3$, there is $t(p)>0$ such that $P_3(p):=\phi_3^-(t(p),p)\in \Sigma_I^3$.
\end{itemize}

Let $K^3$ be the compact region delimited by
\[
\partial K^3=\Sigma_D^3\cup\Sigma_U^3\cup T_1 \cup \cdots \cup T_5,
\]
 where
\[
\begin{aligned}
T_1=&\{(u,0), 0<u\le 13/6\},\\
T_2=&\{(u,v): v =3 u -13/2: u_{f,3}^D \le u \le 13/6\},\\
T_{3}=&\{(u,v): v =4 \sqrt{2} |u|^3/3 -2^{21/8}/3: 0 \le u \le u_{i,3}^D\}, \\
T_{4}=&\{(u,v): v =4 \sqrt{2} |u|^3/3 -2^{21/8}/3: u_{i,3}^I \le u \le 0\}, \\
T_5=&\{(u,v): v =4 \sqrt{2} |u|^3/3 :u_{f,3}^I \le  u \le  0\}.
\end{aligned}
\]

\begin{figure}[H]
\begin{overpic}[width=7cm]{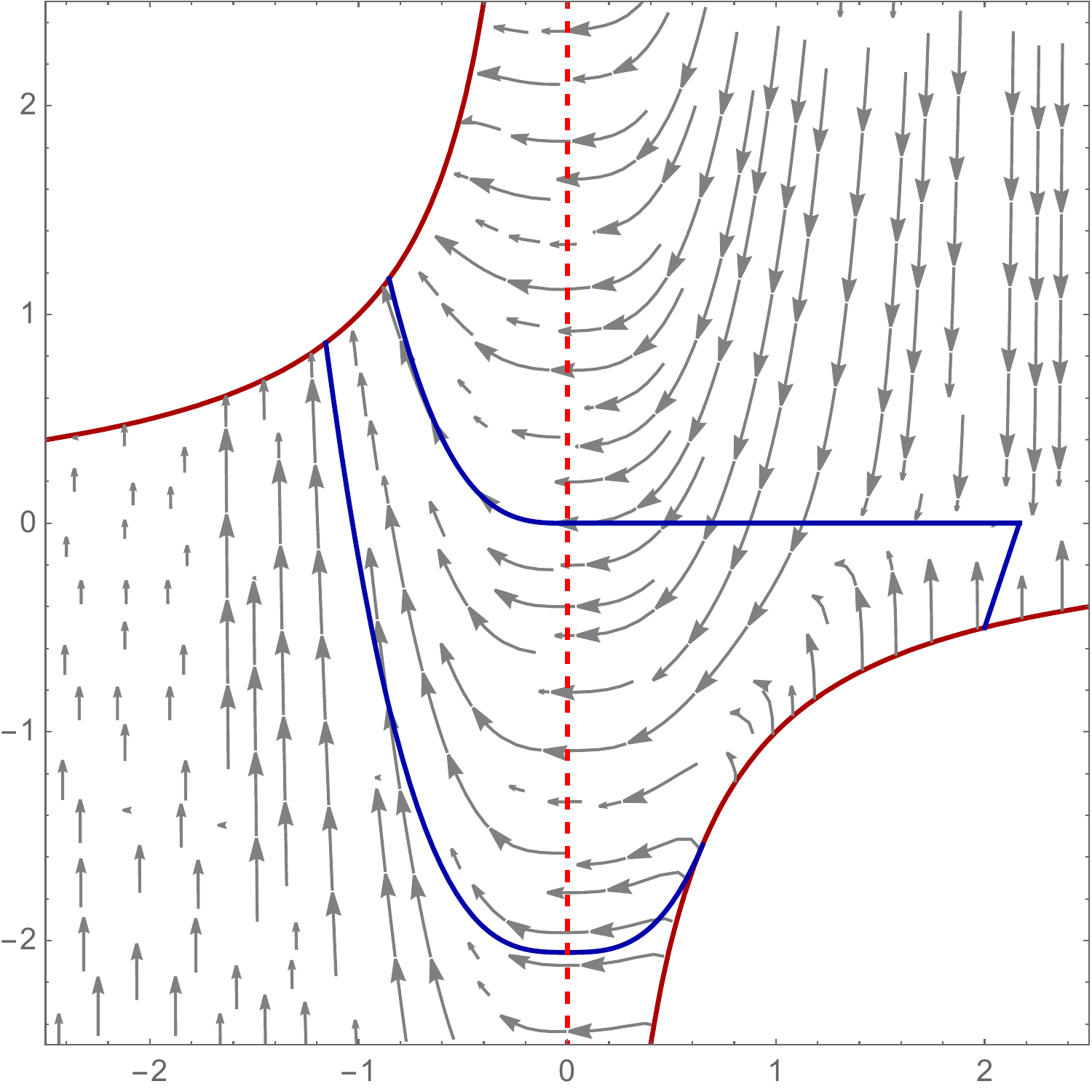}
\put(70,55){$T_1$}
\put(94,46){$T_2$}
\put(70,28){$\Sigma_D^3$}
\put(56,10){$T_3$}
\put(31,25){$T_4$}
\put(27,74){$\Sigma_I^3$}
\put(40,65){$T_5$}
\put(48,35){$K^3$}
	\end{overpic}
\caption{Phase space of the vector field  $\widetilde X_{3}^-$ defined for $u v\geq-1$ and the trapping region $K^3$.}\label{trapping3}
\end{figure}
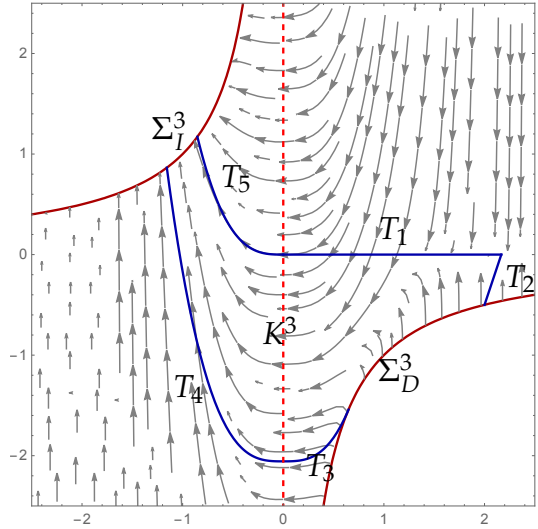

In what follows, we are going to analyze the behavior of the vector field $\widetilde X_{3}^-$ on each component of the boundary of $K^3$.

\noindent{\bf Behavior of $\widetilde X_3^{-}$ on $\Sigma_D^3$ and $\Sigma_I^3$.}
The hyperbola $g(u,v)=u v + 1=0$ satisfies
\[
\langle\nabla g(u,v),\widetilde X_3^{-}\rangle|_{(u,v)\in\Sigma_D^3} = u^3,
\]
which does not vanish on $\Sigma_D^3$.  Note that the flow of the vector field $\widetilde X_3^{-}$ points inwards $K^3$ on $\Sigma_D^3$, because for instance $\widetilde X_1^-(1,-1)=(0,1)$.

The curve $g(u,v)=u v + 1=0$ on $\Sigma_I^3$ satisfies
\[
\langle\nabla g(u,v),\widetilde X_3^{-}\rangle|_{(u,v)\in \Sigma_I^3} =u^3 (1 + 16 u^8),
\]
which does not vanish on $\Sigma_I^3$.  Note that the flow of the vector field $\widetilde X_3^{-}$ points outwards $K^3$ on $\Sigma_I^3$, because for instance $\widetilde X_3^-(-1,1)=(-2,15)$.

\bigskip

\noindent{\bf Behavior of $\widetilde X_3^{-}$ on $T_1$.} The curve $g(u,v)=v=0$ on $u>0$ satisfies
\[
\begin{split}
\langle\nabla g(u,v),\widetilde X_3^{-}\rangle|_{g(u,v)=0} &  =- 8 u^2 \big((1+u^8) - u^4 \sqrt{2 + u^8} \big).
\end{split}
\]
We will show that this derivative does not vanish on $u>0$. We proceed by contradiction. Assume that $6 u (1+u^6 -u^3 \sqrt{2 + u^6})=0$. Then, since $u > 0$ we have $(1+u^8) - u^4 \sqrt{2 + u^8}=0$, and so
\[
\sqrt{2 + u^8}= \frac{1+u^8}{u^4} \ \Leftrightarrow \ 2 + u^8 = \left(\frac{1+u ^8}{u^4}\right)^2 = \frac{1+u^{16} + 2 u^8}{u^8} \ \Leftrightarrow \ \frac{1+4 u^8 + 2 u^{16}}{u^8}=0,
\]
which is not possible. This shows that the flow of the vector field  $\widetilde X_3^{-}$ points inwards $K^3$ on $T_1$, because for instance $\widetilde X_2^-(1,0)=(1-\sqrt{3},-16+8\sqrt{3})$.

\bigskip

\noindent{\bf Behavior of $\widetilde X_3^{-}$ on $T_2$.}
The curve $g(u,v)=v-3 u +13/6=0$  satisfies
\[
\langle\nabla g(u,v),\widetilde X_3^{-}\rangle|_{g(u,v)=0} = -8 u^2 +\frac{117}2 u^3 -30 u^4 -8 u^{10} + (8 u^6+3)\sqrt{2 + u (-13 + 6 u + u^7)}.
\]
We will show that this derivative does not vanish on $T_2$, i.e. for $u_{f,3}^D \le u \le 13/6$. We proceed by contradiction. Assume that it is zero at some $u$. Then, proceeding as in $T_1$ to get rid of the square root, we obtain
\[
p_{10}(u) := 72 - 468 u + 216 u^2 - 256 u^4 + 3744 u^5 - 15225 u^6 + 11544 u^7 - 2412 u^8 + 416 u^{13} - 192 u^{14}.
\]
Using the Sturm procedure we prove that the polynomial $p_{10}(u)$  does not vanish on $u_{f,3}^D  < u < 13/6$ which is a contradiction (see Appendix). Hence, the flow of the vector field  $\widetilde X_3^{-}$ points inwards $K^3$ on $T_2.$

\bigskip

\noindent{\bf Behavior of $\widetilde X_3^{-}$ on $T_3$ and $T_4$.}
On the curve $g(u,v)=v-4 \sqrt{2} |u|^3/3 -2^{21/8}/3=0$ for $u>0$ we have
\[
\begin{split}
\langle\nabla g(u,v),\widetilde X_3^{-}\rangle|_{g(u,v)=0} &=p_{11}(u): =-u^2 \Big(8 - 3 \cdot 2^{21/8} u +16 \sqrt{2} u^4 + 8 u^8 \\
& \phantom{\le} - 4 ( \sqrt{2}+2 u^4) \sqrt{2 - \frac{ 2^{29/8}}3  u + \frac{8 \sqrt{2}}3 u^4 + u^8} \Big),
\end{split}
\]
and for $u<0$,
\[
\begin{split}
\langle\nabla g(u,v),\widetilde X_3^{-}\rangle|_{g(u,v)=0} &=p_{12}(u): =u^2 \Big(8 - 3 \cdot 2^{21/8} u -16 \sqrt{2} u^4 + 8 u^8 \\
& \phantom{\le} - 4 ( \sqrt{2}-2 u^4) \sqrt{2 - \frac{ 2^{29/8}}3  u - \frac{8 \sqrt{2}}3 u^4 + u^8} \Big).
\end{split}
\]
Of course on $u=0$ the above derivative is zero but since it is a quadratic tangency it is enough to show that $p_{11}(u)$ does not vanish  on $0 < u \le u_{i,3}^D$, and that the $p_{12}(u)$ does not vanish on $u_{i,3}^I \le u  <0$.

We proceed by contradiction. Assume first that $p_{11}(u)$ has a zero. Then, proceeding as in  $R_1$ to get rid of the square root, we have that $p_{11}(u)=0$ implies
\[
p_{13}(u)=20 \cdot 2^{5/8} - 54 \cdot 2^{1/4} u - 8 \sqrt{2} u^3 + 80 \cdot 2^{1/8} u^4 - 26 u^7 + 4 \cdot 2^{5/8}  u^8 - 4 \sqrt{2} u^{11}=0.
\]
Using the Sturm procedure  we obtain that the polynomial $p_{13}(u)$  has a unique simple real zero on $u>0$ (see  Appendix). Moreover, we have that
\[
p_{11}(0)=0, \quad   p_{11}'(0)=0, \quad p_{11}''(0)=0, \quad  p_{11}'''(0)=40 \cdot 2^{5/8}, \quad p_1(u_{i,3}^D) =0.420448...
\]
These computations show that if $p_{11}(u)$ has a zero in $(0,u_{i,3}^D)$, then it has another (or a multiple) zero on that interval. This would produce another zero of $p_{13}(u)$ in $(0, \infty)$, which is a contradiction. Hence the flow of the vector field  $\widetilde X_3^{-}$ points inwards $K^3$ on  $T_3$.

On the other hand assume that $p_{12}(u)$ has a zero. Then, proceeding as in $T_1$ to get rid of the square root, we have that $p_{12}(u)=0$ implies
\[
p_{14}(u)=20 \cdot 2^{5/8} - 54 \cdot 2^{1/4} u + 8 \sqrt{2} u^3 - 80 \cdot 2^{1/8} u^4 - 26 u^7 + 4 \cdot 2^{5/8} u^8 + 4 \sqrt{2} u^{11}=0.
\]
Using the Sturm procedure  we have that the polynomial $p_{14}(u)$  has a unique simple real zero on $u<0$ (see  Appendix). Moreover, we have that
\[
p_{12}(0)=0, \quad   p_{12}'(0)=0, \quad p_{12}''(0)=0, \quad p_2'''(0)=-40 \cdot 2^{5/8}, \quad    p_2(u_{i,3}^I)= 40.3062...
\]
These computations show that if $p_{12}(u)$ has a zero in $(u_{i,3}^I,0)$, then either it is multiple or $p_{12}(u)$ has at least two negative zeros in contradiction with uniqueness of negative zeros provided by Sturm procedure. So the flow of the vector field  $\widetilde X_3^{-}$ points inwards $K^3$ on  $T_4$.

\bigskip

\noindent{\bf Behavior of $\widetilde X_3^{-}$ on $T_5$.}
On the curve $g(u,v)=v- 4 \sqrt{2} |u|^3/3 =0$ we have
\[
\langle\nabla g(u,v),\widetilde X_3^{-}\rangle|_{g(u,v)=0}  = 4u^2\left(2 \left(u^8-2 \sqrt{2} u^4+1\right) +\left(2 u^4- \sqrt{2}\right) \sqrt{u^8-\frac{8 \sqrt{2} u^4}{3}+2}\right).
\]
Of course on $u=0$ the above derivative is zero but since it is a quadratic tangency it is to show that this derivative does not vanish on $u_{f,3}^I < u < 0$. Proceeding as in $T_1$ to get rid of the square root, we have that $p_{15}(u):=6-12 \sqrt{2} u^4 + 6 u^8 (\sqrt{6}-2\sqrt{3} u^4)  \sqrt{6 - 8 \sqrt{2} u^4 + 3 u^8}=0$ implies
\[
\frac{2 u^4 (4 \sqrt{2} - 13 u^4 + 2 \sqrt{2} u^8)}{3 (\sqrt{2} - 2 u^4)^2}=0.
\]
Note that the denominator does not vanish on $ u_{f,3}^I < u < 0$, because the unique negative real solution of $\sqrt{2} - 2 u^4=0$, which is $-1/2^{1/8}$ is outside the mentioned interval. On the other hand the unique two real solutions of the numerator are
\[
u_1=-\frac 1 {2^{3/4}} \big(13 \sqrt{2} -  \sqrt{210}\big)^{1/4}\quad \text{and} \quad  u_2=-\frac 1 {2^{3/4}} \big(13 \sqrt{2} + \sqrt{210}\big)^{1/4}.
\]
We observe that $u_1$ belongs to the interval $ u_{f,3}^I < u < 0$ while $u_2$ does not. However, $u_1$ is not a solution of $p_{15}(u)=0$ because
\[
p_{15}(u_1) = -\frac 1 {2^{3/4}}  \big(77 \sqrt{5} -37 \sqrt{21} \big) \ne 0.
\]
This contradiction shows that the flow $\widetilde X_3^-$ point inwards $K^3$ on $T_5$.

\bigskip

\noindent{\bf Existence of $p_3^*$.} Consider the map $h_3 \colon \Sigma_D^3 \to \R$ defined by $h_3(p)=\pi (P_3(p)+p)$, where $\pi:\R^2\to\R$ is the projection onto the first coordinate (see (C3)). Note that by the continuous dependence of the flow of $\widetilde X_3^-$ with respect to the initial conditions the map $h_3$ is continuous. Moreover, the image of the point $(u_{f,3}^D,-1/u_{f,3}^D)$ by $P_3$ (which is the point $w_6$ in Figure \ref{fig.3q3}) is inside $\Sigma_I^3$, and  its symmetric is below its image because $u_{f,3}^D + u_{i,3}^I>0$. Hence $h_3(u_{f,3}^D,-1/u_{f,3}^D)>0$. On the other hand, the image of the point $(u_{i,3}^D,-1/u_{i,3}^D)$ by $P_3$ (which is the point $w_5$ in Figure \ref{fig.3q3}) is inside $\Sigma_I^3$, and  its symmetric is above its image because  $u_{i,3}^D +u_{f,3}^I<0$. Therefore $h_2(u_{i,3}^D,-1/u_{i,3}^D)<0$. Thus, by continuity, there exists $p_3^*\in\Sigma_D^2$ such that $h_3(p_3^*)=0$, i.e, $P_3(p_3^*)=-p_3^*$ and so $ \phi_3^-(t^*_3,p_3^*)=-p_3^*$ as we wanted to prove.

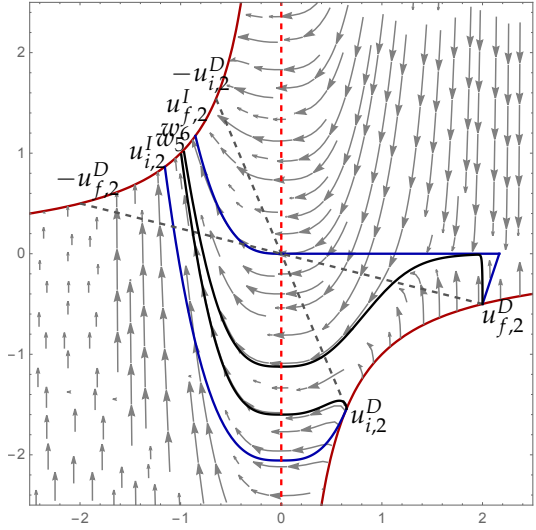
\begin{figure}[H]
\begin{overpic}[width=7cm]{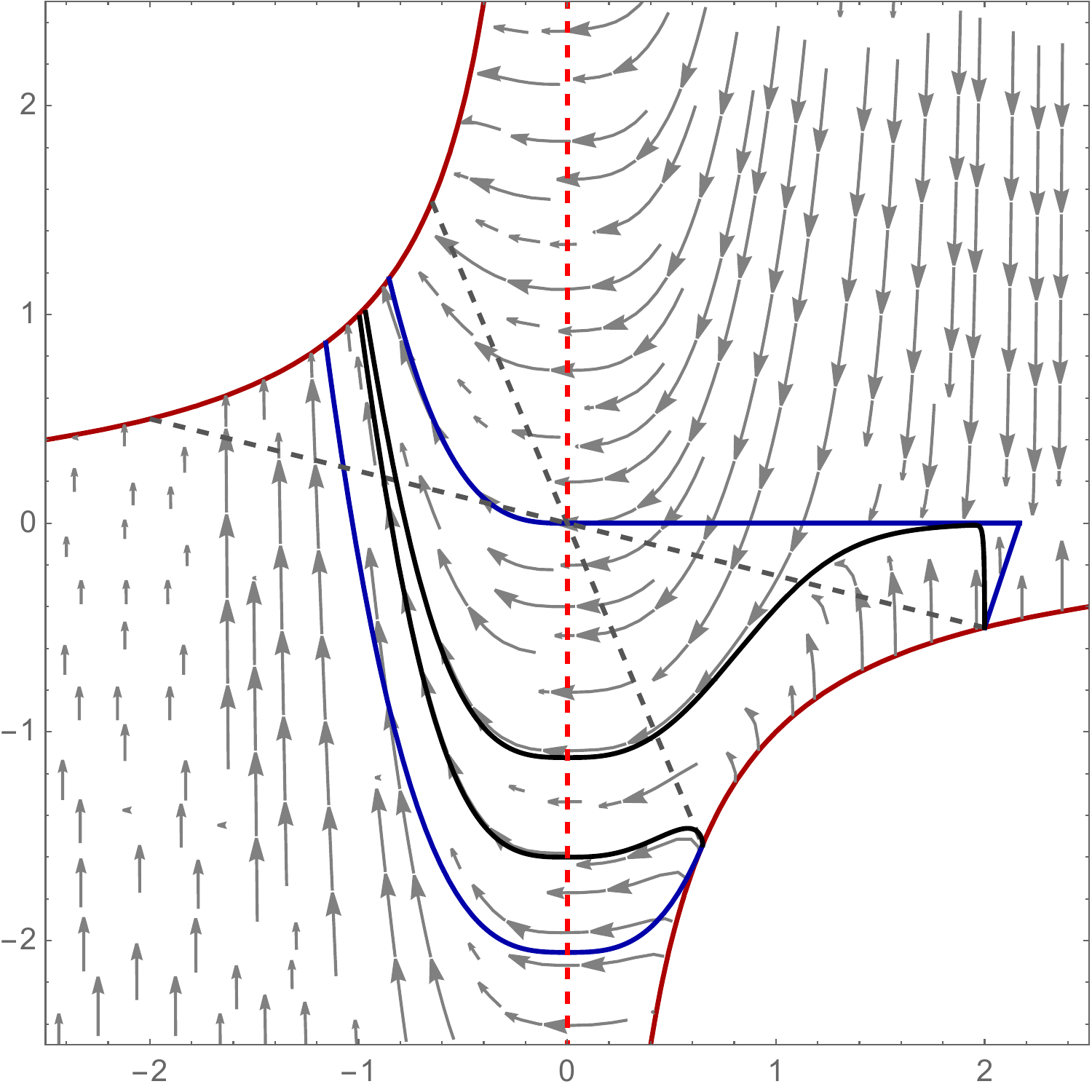}
\put(29,74.5){\footnotesize{$w_6$}}
\put(28,72.5){\footnotesize{$w_5$}}
\put(9,65){\footnotesize{$-u_{f,2}^D$}}
\put(31,85){\footnotesize{$-u_{i,2}^D$}}
\put(30,79){\footnotesize{$u_{f,2}^I$}}
\put(23,70){\footnotesize{$u_{i,2}^I$}}
\put(90,39){\footnotesize{$u_{f,2}^D$}}
\put(65,20){\footnotesize{$u_{i,2}^D$}}
	\end{overpic}
\caption{Idea of the obtention of the fixed point: $w_5 =P_3(u_{i,3}^D)$ and $w_6=P_3(u_{f,3}^D)$.}\label{fig.3q3}
\end{figure}

\section{Proof of Theorem \ref{main-general}}\label{sec:q-arbitrary}

This section is devoted to the proof of Theorem \ref{main-general}. This proof will follow from propositions \ref{anlyitcproof}, \ref{proofU2}, \ref{proofU3}, and  \ref{proofU4}.

Consider the sections $\Sigma_D^q$ and $\Sigma_I^q$ on the hyperbola $u\,v+1=0,$ given by
\[
\begin{aligned}
\Sigma_D^q=&\{(u,v): v =-1/u: u_{i,q}^D \le u\le u_{f,q}^D\}\quad \text{and}\\
\Sigma_I^q=&\{(u,v): v =-1/u:  u_{i,q}^I \le u\le u_{f,q}^I\},
\end{aligned}
\]
where
\[
u_{i,q}^D=2^{-\frac{1}{2(1+q)}}(1+q)^{-\frac{1}{1+q}}, \quad  u_{f,q}^D= 2,
\quad  u_{f,q}^I=-\left(\frac{\sqrt{2}(1+q)}{q}\right)^{-1/(1+q)},
\]
and $ u_{i,q}^I=-x_{i,q}^I$, being $x_{i,q}^I$ the unique positive real root of the polynomial
\[
P(x)=-q - 2^{\frac{1}{2(1+q)}}(1+q)^{\frac{2+q}{1+q}} x + \sqrt{2} (1+q) x^{1+q}.
\]
Indeed, in the next result we show that $P(x)$ has a unique positive real zero $x_{i,q}^I$ and that $x_{i,q}^I <  2$.

\begin{lemma}
The polynomial $P(x)$ has a unique positive real zero $x_{i,q}^I$. In addition, $x_{i,q}^I < u_{f,q}^D= 2$.
\end{lemma}

\begin{proof}
First, from Descartes rule of signs, $P(x)$ has at most one positive real zero. Let us see that it has exactly one positive real zero. Indeed, since
\[
P'(x)=- 2^{\frac{1}{2(1+q)}}(1+q)^{\frac{2+q}{1+q}}  + \sqrt{2} (1+q)^2 x^{q}
\]
and, for $x > 2$, $x^q > 2^{\frac{1}{2(1+q)}}$ and $(1+q)^2 > (1+q)^{\frac{2+q}{1+q}}$, it follows that $P(x)$ is increasing for $x > 2$. Since $P(0)=-q < 0$, then $P(x)$ has exactly one positive real zero $x_{i,q}^I$.

Now, let us check that $x_{i,q}^I<2$ by  showing that $P(2)>0$. Note that
\[
P(2)=-q + 2^{\frac{3+2q}{2}} (1+q)-2^{\frac{3+2q}{2(1+q)}} (1+q)^{\frac{2+q}{1+q}}.
\]
First, for $q=1$, $P(2)=-1 - 8 2^{1/6} + 8 \sqrt{2}>0$. Now, for $q \ge 2$, one can see that  $P(2)>0 $ if, and only if,
\begin{equation}\label{eq:efe}
 \Big(2^q-\frac{q}{2^{\frac{3}2}(q+1)} \Big)^{q+1}  > 2^{-q/2} (q+1).
\end{equation}
We claim
\[
2^q-\frac{q}{2^{\frac{3}2}(q+1)} > 2 \quad \text{and} \quad 2^{q+1}>2^{-q/2} (q+1).
\]
The relation in \eqref{eq:efe} will then follow from the claim.

In order to get the claim, we compute the derivative in the variable $q$  as follows:
\[
\dfrac{d}{dq}\Big(2^q-\frac{q}{2^{\frac{3}2}(q+1)} -2\Big) = \frac{q}{32(1+q)^2} -\frac{1}{32(1+q)} +2^q \log 2 > \frac{q}{32(1+q)^2}>0
\]
for any $q \ge 2$. Since, for $q=2$, we get
\[
2^2-\frac{2}{2^{\frac{3}2}3} -2 =1.97>0,
\]
the first assertion in the claim follows. For the second one, taking again derivatives in $q$ , we obtain
\[
\dfrac{d}{dq}\big(2^{q+1}-2^{-q/2} (q+1)\big)= 2^{1+q} \log 2 - 2^{-\frac{2+q}2} (2-(1+q)\log 2).
\]
Note that, for $q > 1$, we have $2-(1+q)\log 2  < \log 2$ and, since $2^{-\frac{2+q}2} < 2^{1+q}$, it follows that the derivative above is positive for $q >1$. Finally, since for $q=1$ we get $2^2-2^{1/2}>0$, then the second assertion of the claim follows. Hence, the proof of the lemma is concluded. 
\end{proof}

From here, by following the strategy developed for the cases $q=1,2,3$, the idea is to show that the flow of  $\widetilde X_q^-$ induces a transition map
\begin{equation}\label{eq:P3}
P_q:\Sigma_D^q\to\Sigma_I^q.
\end{equation}
As before, this will be done by constructing a region $K^q$ such that $\Sigma_D^q\cup\Sigma_I^q\subset \partial K^q$ and $\widetilde X_q^-$ points inwards $K^q$ everywhere on $\partial K^q$ except at $\Sigma_I^q$ (see Figure \ref{trapping4}). Since $\widetilde X_q^-$ does not have singularities in $K^q$, from the Poincar\'{e}-Bendixson Theorem (see, for instance, \cite[Theorem 1.25]{DLA06}), we  conclude that for each point $p\in\Sigma_D^q$, there is $t(p)>0$ such that $P_q(p):=\phi_q^-(t(p),p)\in \Sigma_I^q$.

By imitating the construction of the compact regions in cases $q=1,2,3$, we consider $K^q$ the compact region (see Figure \ref{trapping4}) delimited by
\[
\partial K^q=\Sigma_D^q\cup\Sigma_I^q\cup U_1 \cup \cdots \cup U_5,
\]
 where
\[
\begin{aligned}
U_1=&\{(u,0), 0<u\le (1+4q)/(2q)\},\\
U_2=&\{(u,v): v =q u -(1+4q)/2: u_{f,q}^D < u \le (1+4q)/(2q)\},\\
U_{3}=&\{(u,v): v =(1+q) \sqrt{2} |u|^q/q -2^{\frac{1}{2(1+q)}} (1+q)^{\frac{2+q}{1+q}}/q: 0 \le u < u_{i,q}^D\}, \\
U_{4}=&\{(u,v): v =(1+q) \sqrt{2} |u|^q/q -2^{\frac{1}{2(1+q)}} (1+q)^{\frac{2+q}{1+q}}/q: u_{i,q}^I < u \le 0\}, \\
U_5=&\{(u,v): v =(1+q) \sqrt{2} |u|^q/q :u_{f,q}^I \le  u \le  0\}.
\end{aligned}
\]

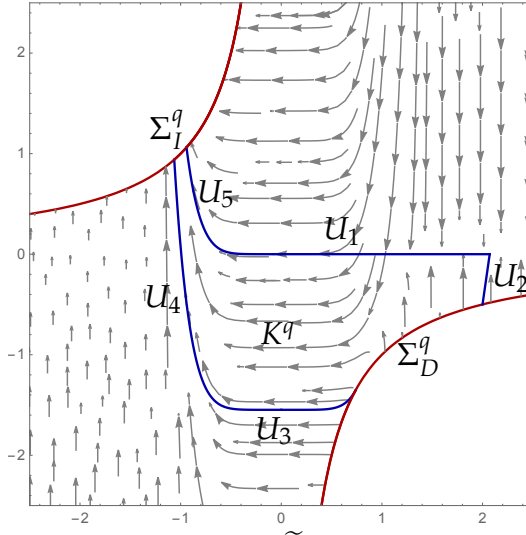
\begin{figure}[H]
\begin{overpic}[width=7cm]{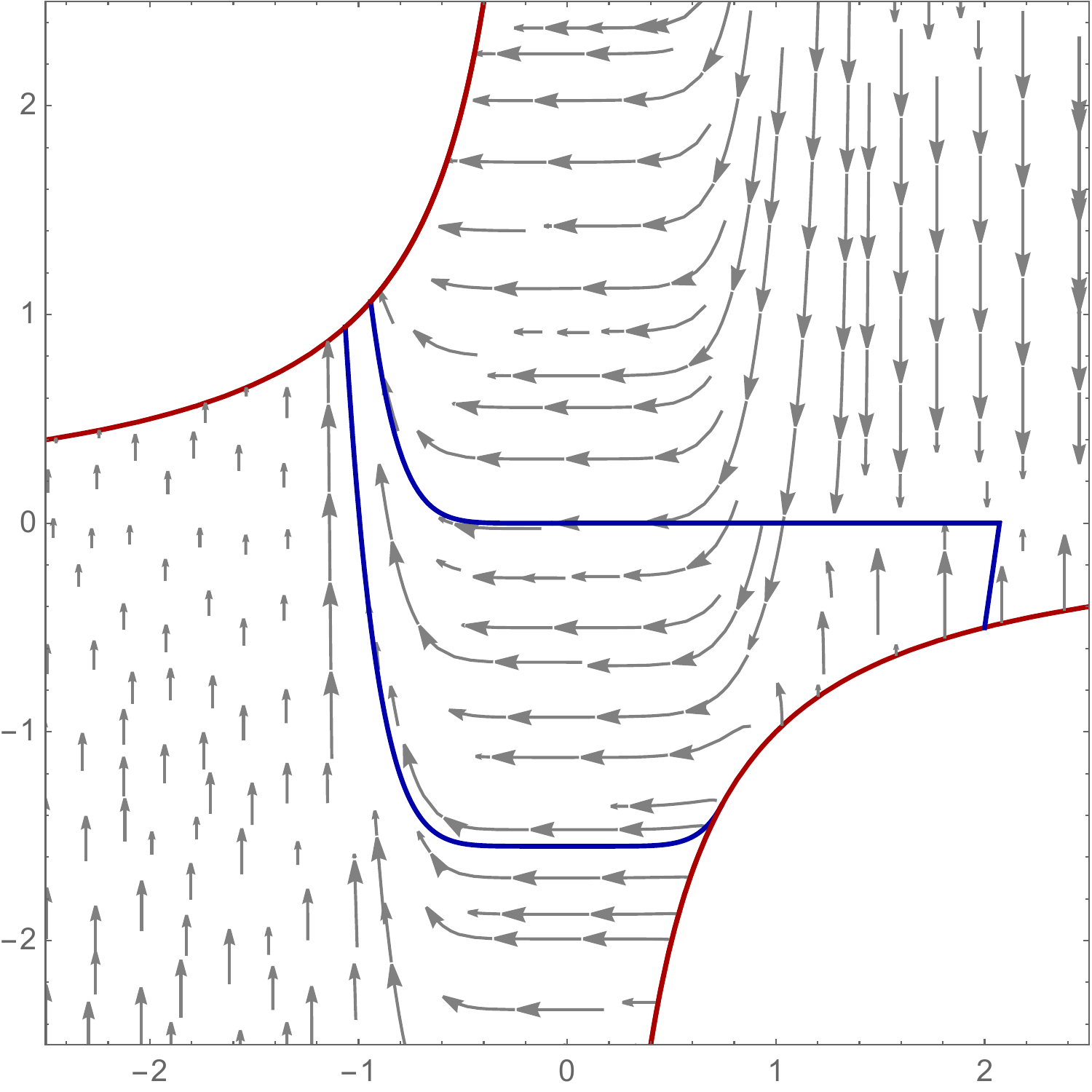}
\put(60,55){$U_1$}
\put(92,46){$U_2$}
\put(74,31){$\Sigma_D^q$}
\put(47,17){$U_3$}
\put(26,42){$U_4$}
\put(27,74){$\Sigma_I^q$}
\put(36,62){$U_5$}
\put(48,35){$K^q$}
	\end{overpic}
\caption{Phase space of the vector field  $\widetilde X_{q}^-$ defined for $u v\geq-1$ and the region $K^q$ for $q=7$.}\label{trapping4}
\end{figure}

In what follows, we are going to analyze the behavior of the vector field $\widetilde X_{q}^-$ on each component of the boundary of $K^q$. Summarizing, in Proposition \ref{anlyitcproof}, we will prove analytically that, for every positive integer $q$ , the flow of the vector field $\widetilde X_q^{-}$ points inwards $K^q$ on $\Sigma_D^q$, $U_1$, and $U_5$, and outwards $K^q$ on $\Sigma_I^q$; for the remaining curves, $U_2$, $U_3$, and $U_4$, propositions \ref{proofU2}, \ref{proofU3}, and  \ref{proofU4} will establish, respectively, that, for every positive integer $q$ , $\widetilde X_q^{-}$ points inwards $K^q$ on them provided that the conditions {\bf C1},  {\bf C2}, and {\bf C3} of Theorem \ref{main-general} hold.

\begin{proposition}\label{anlyitcproof}
The flow of the vector field $\widetilde X_q^{-}$ points inwards $K^q$ on $\Sigma_D^q$, $U_1$, and $U_5$, and outwards $K^q$ on $\Sigma_I^q$.
\end{proposition}

\begin{proof}
We first prove the proposition on $\Sigma_D^q$ and $\Sigma_I^q$. Note that the hyperbola $g(u,v)=u v + 1=0$ satisfies
\[
\begin{aligned}
&\langle\nabla g(u,v),\widetilde X_q^{-}\rangle|_{(u,v)\in\Sigma_D^q} = u^q,\\
&\langle\nabla g(u,v),\widetilde X_q^{-}\rangle|_{(u,v)\in \Sigma_I^q} =(-1)^{q+1}(1 + 4 (1 + q) u^{2(1+q)}) u^q,
\end{aligned}
\]
which does not vanish on $\Sigma_D^q$ and $\Sigma_I^q$, respectively. Since, for instance,  $\widetilde X_q^-(1,-1)=(0,1)$ and $\widetilde X_q^-(-1,1)=(-2,3+4q)$, we conclude that the flow of the vector field $\widetilde X_q^{-}$ points inwards $K^q$ on $\Sigma_D^q$ and $\Sigma_I^q$.

Now, consider the curve $U_1$ described by $g(u,v)=v=0$, $u>0$. Note that
\[
\begin{split}
\langle\nabla g(u,v),\widetilde X_q^{-}\rangle|_{g(u,v)=0} &  =-2 (1 + q) u^{q-1} \big(1 + u^{2(1+q)} - u^{1+q} \sqrt{2 + u^{2(1+q)}}\big).
\end{split}
\]
We will show that this derivative does not vanish on $u>0$. We proceed by contradiction. Assume that $1 + u^{2(1+q)} - u^{1+q} \sqrt{2 + u^{2(1+q)}}=0$. Then,
\[
\sqrt{2 + u^{2(1+q)}}= \frac{1+u^{2(1+q)}}{u^{1+q}} \ \Leftrightarrow \ 2 + u^{2(1+q)} = \left(\frac{1+u^{2(1+q)}}{u^{1+q}}\right)^2\ \Leftrightarrow \  -\frac{1}{u^{2(1+q)}}=0,
\]
which is not possible. Thus, the flow of the vector field  $\widetilde X_q^{-}$ points inwards $K^q$ on $U_1$.

Finally, consider the curve $U_5$ described by $g(u,v)=v- (1+q) \sqrt{2} |u|^q/q =0$ for  $u_{f,q}^I \le  u \le  0$ . Note that
\[
\begin{aligned}
\langle\nabla g(u,v),\widetilde X_q^{-}\rangle& |_{g(u,v)=0}   =
\frac{\sqrt{2}}q (1+q) (-1)^{q+1} u^{q-1} \Big(\sqrt{2} q + 3 (-1)^q  (1 + q)  u^{1 + q} \\
    &+
    \sqrt{2} q  u^{2(1 + q)}  - \sqrt{q} ( 1 + \sqrt{2}(-1)^q u^{1 + q}) \sqrt{
  2 q + 2 (-1)^q \sqrt{2} (1+q) u^{1 + q} +
   q u^{2(1 + q)}}\Big).
   \end{aligned}
\]
Of course on $u=0$ the above derivative is zero but, since it is a quadratic tangency, it is enough to show that this derivative does not vanish on $u_{f,q}^I < u < 0$. Proceeding as in $U_1$ to get rid of the square root, we have that
\[
p_{0}(u):=q (-1)^{q+1} \frac{\langle\nabla g(u,v),\widetilde X_q^{-}\rangle|_{g(u,v)=0}}{\sqrt{2} (q+1) u^{q-1}} =0
\]
implies
\[
\frac{2 u^{q+1} (-4 (-1)^q \sqrt{2} q - (9 + 10 q) u^{1 + q} +
 2 (-1)^{1+q} \sqrt{2} q u^{2(1 + q)})}{2 q^2 (1 +\sqrt{2}(-1)^q  u^{1 + q})^2}=0.
\]
Note that the denominator does not vanish on $ u_{f,3}^I < u < 0$, because the unique negative real solution of $1 +\sqrt{2}(-1)^q  u^{q+1}=0$, namely $u=-2^{-1/(2(1+q))},$ is outside the mentioned interval.
On the other hand, the unique two real solutions of the numerator are
\[
u_{\pm} =\left(\frac{18 + 20 q \pm 6 \sqrt{9 + 20 q + 4 q^2}}{8 \sqrt{2} q}\right)^{1/(1+q)}.
\]
We observe that $u_-$ belongs to the interval $ u_{f,q}^I < u < 0$ while $u_+$ does not. However, $u_-$ is not a solution of $p_{0}(u)=0$ because
\[
p_{0}(u_-) = -32 q (1 + 2 q) \left(1 + 648 q + 288 q^2 + 32 q^3 +
  ( 27 + 6q) \sqrt{2(1+2q)} \right) \ne 0.
\]
Therefore, $p_0(u)$ does not vanish for  $ u_{f,q}^I < u < 0$, which implies that the flow $\widetilde X_q^-$ points inwards $K^q$ on $U_5$. It concludes the proof of this proposition.
\end{proof}

\begin{proposition}\label{proofU2}
The vector field  $\widetilde X_q^{-}$ points inwards $K^q$ on $U_2$ provided that the polynomial $P_0$, defined in \eqref{polynomials}, does not vanish on  $I_0=\big(u_{f,q}^D\,,\,(1+4q)/(2q)\big)$.
\end{proposition}
\begin{proof}
 Consider $U_2$ described by  $g(u,v):=v-q u +(1+4q)/2=0,$  for $u\in I_0$.  Denote $p_0(u)= \langle\nabla g(u,v),\widetilde X_q^{-}\rangle|_{g(u,v)=0}$, $u\in I_0$. Notice that
\[
\begin{aligned}
p_0(u) = & (q +2 (1+q) u^{2q})\sqrt{2 - u - 4 q u + 2 q u^2 + u^{2(1+q)}}  \\
&+ \frac 1 2  u^{q-1} (4(1+q)  - (3+2q) (1+4q) u + 4 q (2+q)u^2  +
   4 (1+q) u^{2(1+q)}).
\end{aligned}
\]
We claim that $p_0(u)$ does not vanish for $u\in I_0$. Indeed, if $p_0(u^*)=0$ at some $u^*\in I_0$, then, by proceeding as in $U_1$ to get rid of the square root, we obtain that $P_0(u^*)=0$, which contradicts the hypothesis. Hence, the flow of the vector field  $\widetilde X_q^{-}$ points inwards $K^q$ on $U_2$.
\end{proof}

\begin{proposition}\label{proofU3}
The vector field  $\widetilde X_q^{-}$ points inwards $K^q$ on $U_3$ provided that the polynomial $P^+ $, defined in \eqref{polynomials},  has at most one root (counting its multiplicity) in $I^+ =\big(0\,,\, u_{i,q}^D\big)$.
\end{proposition}
\begin{proof}

Consider $U_3$ described by $g(u,v)=v -(1+q) \sqrt{2} |u|^q/q -2^{\frac{1}{2(1+q)}} (1+q)^{\frac{2+q}{1+q}}/q=0$ for $0 \le u \le u_{i,q}^D$. Note  that
\[
\begin{aligned}
p_{1}(u): =& \langle\nabla g(u,v),\widetilde X_q^{-}\rangle|_{g(u,v)=0}  \\
=&\frac{1+q}{q} u^{q-1} \Big( \sqrt{q} (\sqrt{2} + 2 u^{1+q}) \sqrt{R}
 \\
  & - 2 q + 2^{\frac{1}{2(1+q)}} (1 + q)^{\frac{1}{1+q}} (3 + 2 q) u -
   3 \sqrt{2} (1 + q) u^{1+q} - 2 q u^{2(1+q)}
   \Big),
   \end{aligned}
   \]
where $R=2 q - 2^{\frac{3+2q}{2(1+q)}} (1+q)^{\frac{2+q}{1+q}} u +
     2 \sqrt{2} (1 + q) u^{1+q} + q u^{2(1+q)}$.

Of course on $u=0$ the above derivative is zero but since it is a quadratic tangency it is enough to show that $p_{1}(u)$ does not vanish  on $0 < u \le u_{i,q}^D$. We proceed by contradiction. Assume that $p_{1}(u)$ has a zero. Then, proceeding as in  $R_1$ to get rid of the square root, we have that $p_{1}(u)=0$ implies $P^+ (u)=0$.  Notice that
\[
p_{1}(0)=p_{1}'(0)= \cdots= p_{1}^{(q-1)}(0)=0, \quad  p_{1}^{(q)}(0)=\frac{2+q}{q} 2^{\frac{1}{2(1+q)}} (1+q)^{\frac{2+q}{1+q}} >0,
\]
and $ p_1(u_{i,q}^D) =(u_{i,q}^D)^{q-1} >0$. This implies that, if $p_{1}(u)$ has a zero in $(0,u_{i,q}^D)$, then it has another (or a multiple) zero on that interval. This would produce two zeros, counting their multiplicity, of $P^+ (u)$ in $(0,u_{i,q}^D)$, which contradicts the hypothesis.

Hence the flow of the vector field  $\widetilde X_q^{-}$ points inwards $K^q$ on $U_3$.
\end{proof}

\begin{proposition}\label{proofU4}
The vector field  $\widetilde X_q^{-}$ points inwards $K^q$ on $U_4$ provided that the polynomial $P^- $, defined in \eqref{polynomials}, has at most one root (counting its multiplicity) in $I^-:=(-2,0)$
\end{proposition}

\begin{proof}

Consider $U_4$ described by $g(u,v)=v -(1+q) \sqrt{2} |u|^q/q -2^{\frac{1}{2(1+q)}} (1+q)^{\frac{2+q}{1+q}}/q=0$ for $u<0$. Notice that
\[
\begin{aligned}
p_{2}(u): =& \langle\nabla g(u,v),\widetilde X_q^{-}\rangle|_{g(u,v)=0}\\
 =&
\frac{1+q}{q} u^{q-1} \Big(\sqrt{q} ((-1)^q\sqrt{2}  + 2 u^{1+q})\sqrt{R} \\
&  -2 q (-1)^{q} + (-1)^{q}2^{\frac{1}{2(1+q)}} (1 + q)^{\frac{1}{1+q}} (3 + 2 q) u -
 3 \sqrt{2} (1 + q) u^{1+q} - 2 (-1)^{q} q u^{2(1+q)} \Big),
\end{aligned}
\]
where $R=2 q - 2^{\frac{3+2q}{2(1+q)}} (1+q)^{\frac{2+q}{1+q}} u +(-1)^q
     2 \sqrt{2} (1 + q) u^{1+q} + q u^{2(1+q)}$.

Of course on $u=0$ the above derivative is zero but since it is a quadratic tangency it is enough to show that $p_{2}(u)$ does not vanish on $u_{i,q}^I \le u  <0$. Analogously to Proposition \ref{proofU3}, we proceed by contradiction. Assume that $p_{2}(u)$ has a zero. Then, proceeding as in $U_1$ to get rid of the square root, we have that $p_{2}(u)=0$ implies $P^- (u)=0$. Notice that
\[
p_{2}(0)=p_{2}'(0)= \cdots= p_{2}^{(q-1)}(0)=0, \quad\text{and}\quad  p_{2}^{(q)}(0)=-\frac{2+q}{q} 2^{\frac{1}{2(1+q)}} (1+q)^{\frac{2+q}{1+q}} <0.
\]
In addition, one can see tha $p_2(-2)>0$.
This implies that if $p_{2}(u)$ has a zero in $[u_{i,q}^I,0)\subset(-2,0)$, then it has another (or a multiple) zero in the interval $(-2,0)$. From here, we get a contradiction with the hypothesis.

Hence the flow of the vector field  $\widetilde X_q^{-}$ points inwards $K^q$ on $U_4$.
\end{proof}

Finally, consider the map $h_q \colon \Sigma_D^q \to \R$ defined by $h_q(p)=\pi (P_q(p)+p)$, where $\pi:\R^2\to\R$ is the projection onto the first coordinate. Note that by the continuous dependence of the flow of $\widetilde X_q^-$ with respect to the initial conditions the map $h_q$ is continuous. Moreover, the image of the point $(u_{f,q}^D,-1/u_{f,q}^D)$ by $P_q$ is inside $\Sigma_I^q$, and  its symmetric is below its image because $u_{f,q}^D + u_{i,q}^I>0$. Hence $h_q(u_{f,q}^D,-1/u_{f,q}^D)>0$. On the other hand, the image of the point $(u_{i,q}^D,-1/u_{i,q}^D)$ by $P_q$ is inside $\Sigma_I^q$, and  its symmetric is above its image because  $u_{i,q}^D +u_{f,q}^I<0$. Therefore $h_q(u_{i,3}^D,-1/u_{i,3}^D)<0$. Thus, by continuity, there exists $p_q^*\in\Sigma_D^q$ such that $h_q(p_q^*)=0$, i.e, $P_q(p_q^*)=-p_q^*$ and so $ \phi_q^-(t^*_q,p_q^*)=-p_q^*$ as we wanted to prove (see Section \ref{sec:setup}).

\section*{Appendix: Sturm procedure for $q\in\{1,2,3\}$}

Let $p(u)$ be a square-free polynomial of degree $d$, and consider the so-called {\it Sturm sequence} $q_i(u),$ for $i=0,\ldots,\ell,$ given by: $q_0(u)=p_0(u)$, $q_1(u)=p'_0(u),$ and for $i=2,\ldots,\ell$, $-q_i(u)$ is the polynomial reminder of the division of  $q_{i-2}$ by $q_{i-1}$, where $\ell$ is the first index for which $q_{\ell}$ is constant. Sturm Theorem (see, for instance, \cite[Theorem 5.6.2]{SB}) provides that the number of real roots of $p(x)$ in the half-open interval $(a,b]$ is $V(a)-V(b)$, where $V(x)$ is the number of sign variation of the sequence $q_0(x),q_1(x)\ldots,q_{\ell}(x)$. Notice that when computing $V(x)$, $x$ can take the value $\infty$ (resp. $-\infty$), in these cases $V(\infty)$ (resp. $V(-\infty)$) denotes the number of sign variation of the leading terms of the sequence $q_0(x),q_1(x)\ldots,q_{\ell}(x)$ (resp. $q_0(-x),q_1(-x)\ldots,q_{\ell}(-x)$).

\bigskip

\subsection{ Sturm Procedure for $q=1$.}

First, note that the polynomial
\[
p_0(u)=-14 + 95 u - \frac{745}4 u^2+ 110 u^3 - 19 u^4 + 20 u^5 - 8 u^6
\]
and its derivative
\[
p_0'(u)=-95 - \frac{745}2 u+ 330 u^2 - 76 u^3 + 100 u^4 - 48 u^5
\]
have no common zeroes because its resultant with respect to $u$ is $-1.52127..\cdot 10^{18}.$ In particular, $p(x)$ is a square-free polynomial and the Sturm procedure can be applied directly. By computing the Sturm sequence $q_i(u)$ for $i=0,\ldots,6,$ we see $V(2)=V(5/2)=2$, which implies that $p_0(u)$ has no real zero in the interval $(2,5/2)$.

Now, notice that the polynomial
\[
p_3(u) =12 \cdot 2^{3/4} - 58 \sqrt{2} u + 88 \cdot 2^{1/4} u^2 - 38 u^3 + 4 \cdot 2^{3/4} u^4 - 4 \sqrt{2}  u^5
\]
has no common zeroes with its derivative because the resultant  between $p_3(u)$ and $p_3'(u)$ with respect to the variable $u$ is $-5637568724992\sqrt{2}$ which is not zero. In particular, $p_3(u)$ is a square-free polynomial and the Sturm procedure can be applied directly. By computing the Sturm sequence $q_i(u)$ for $i=0,\ldots,5,$ we see that $V(0)=3$ and $V(\infty)= 2$. Therefore it follows from the Sturm process that $p_3(u)$ has a unique positive real zero.

Finally, the polynomial
\[
p_4(u) =12 \cdot 2^{3/4} - 42 \sqrt{2} u - 88 \cdot 2^{1/4} u^2 - 38 u^3 + 4 \cdot 2^{3/4} u^4 +  4 \sqrt{2}  u^5
\]
has no common zeroes with its derivative because the resultant  between $p_4(u)$ and $p_4'(u)$ with respect to the variable $u$ is
$-88414837800960\sqrt{2}$ which is not zero. In particular, $p_4(u)$ is a square-free polynomial and the Sturm procedure can be applied directly. By computing the Sturm sequence $q_i(u)$ for $i=0,\ldots,5,$ we see that $V(-\infty)=4$ and $V(0)= 3$. Therefore, it follows from the Sturm process that $p_4(u)$ has a unique negative real zero.

\bigskip

\subsection{ Sturm Procedure for $q=2$}

First, note that the polynomial
\[
p_6(u)=32 - 144 u - 80 u^2 + 1512 u^3 - 4545 u^4 + 3168 u^5 - 624 u^6 + 216 u^9 - 96 u^{10}
\]
has no common zeroes with its derivative because the resultant between $p_6(u)$ and $p_6'(u)$ with respect to $u$ is $-4.41356.. \cdot 10^{57}.$ In particular, $p_6(u)$ is a square-free polynomial and the Sturm procedure can be applied directly. By computing the Sturm sequence $q_i(u)$ for $i=0,\ldots,10,$ we see $V(2)=V(9/4)=3$, which implies that $p_6(u)$ has no real zero in the interval $(2,9/4)$.

Now, notice that the polynomial
\[
p_8(u) =32 \cdot  2^{1/6} 3^{1/3} - 49 \cdot 2^{1/3}  3^{2/3} u - 16 \sqrt{2} u^2 + 78 \cdot 2^{2/3}   3^{1/3} u^3 - 58 u^5 + 8 \cdot 2^{1/6}  3^{1/3} u^6 - 8 \sqrt{2} u^8
\]
has no common zeroes with its derivative because the resultant  between $p_8(u)$ and $p_8'(u)$ with respect to the variable $u$ is
\[
9669300766922659513289932800 \cdot 2^{1/6}  3^{1/3}
\]
which is not zero. In particular, $p_8(u)$ is a square-free polynomial and the Sturm procedure can be applied directly. By computing the Sturm sequence $q_i(u)$ for $i=0,\ldots,10,$ we see that $V(-\infty)= 5$, $V(0)=4$, and $V(\infty)=3$. Therefore, it follows from the Sturm process that $p_8(u)$ has a unique positive real zero and a unique negative real zero.

\bigskip

\subsection{  Sturm Procedure for $q=3$.}
First, note that the polynomial
\[
p_{10}(u) := 72 - 468 u + 216 u^2 - 256 u^4 + 3744 u^5 - 15225 u^6 + 11544 u^7 - 2412 u^8 + 416 u^{13} - 192 u^{14}.
\]
has no common zeroes with its derivative because the resultant between $p_{10}(u)$ and $p_{10}'(u)$ with respect to $u$ is $-2.74875.. \cdot 10^{97}.$ In particular, $p_{10}(u)$ is a square-free polynomial and the Sturm procedure can be applied directly. By computing the Sturm sequence $q_i(u)$ for $i=0,\ldots,10,$ we see $V(2)=V(13/6)=3$, which implies that $p_{10}(u)$ has no real zero in the interval $(2,13/6)$.

Now, notice that the polynomial
\[
p_{13}(u) =20 \cdot 2^{5/8} - 54 \cdot 2^{1/4} u - 8 \sqrt{2} u^3 + 80 \cdot 2^{1/8} u^4 - 26 u^7 + 4 \cdot 2^{5/8}  u^8 - 4 \sqrt{2} u^{11}
\]
has no common zeroes with its derivative because the resultant  between $p_{13}(u)$ and $p_{13}'(u)$ with respect to the variable $u$ is
$-5.12026.. \cdot 10^{35}$ which is not zero. In particular, $p_{13}(u)$ is a square-free polynomial and the Sturm procedure can be applied directly. By computing the Sturm sequence $q_i(u)$ for $i=0,\ldots,11,$ we see that $V(0)=5$ and $V(\infty)= 4$. Therefore it follows from the Sturm process that $p_{13}(u)$ has a unique positive real zero.

Finally, the polynomial
\[
p_{14}(u) =20 \cdot 2^{5/8} - 54 \cdot 2^{1/4} u + 8 \sqrt{2} u^3 - 80 \cdot 2^{1/8} u^4 - 26 u^7 + 4 \cdot 2^{5/8} u^8 + 4 \sqrt{2} u^{11}
\]
has no common zeroes with its derivative because the resultant between $p_{14}(u)$ and $p_{14}'(u)$ with respect to the variable $u$ is
$-2.29037.. \cdot 10^{37}$ which is not zero. In particular, $p_{14}(u)$ is a square-free polynomial and the Sturm procedure can be applied directly. By computing the Sturm sequence $q_i(u)$ for $i=0,\ldots,11,$ we see that $V(-\infty)=7$ and $V(0)= 6$. Therefore, it follows from the Sturm process that $p_{14}(u)$ has a unique negative real zero.

\section{Supplementary Material}

See the supplementary material for a Mathematica algorithm (based on the Sturm procedure explained in the Appendix) that, for a given arbitrary positive integer $q$, checks whether conditions {\bf C1}, {\bf C2}, and {\bf C3} of Theorem \ref{main-general} hold or not  by computing the number of roots of the polynomials $P_0$, $P^+$, and $P^-$ in the intervals $I_0$, $I^+$, and $I^-$, respectively.

\section*{Acknowledgments}
J. Llibre is supported by the Agencia Estatal de Investigaci\'on grant PID2019-104658GB-I00, and the H2020 European Research Council grant MSCA-RISE-2017-777911.
D.D. Novaes is partially supported by S\~{a}o Paulo Research Foundation (FAPESP) grants 2022/09633-5, 2021/10606-0, 2019/10269-3, and 2018/13481-0, and by Conselho Nacional de Desenvolvimento Cient\'{i}fico e Tecnol\'{o}gico (CNPq) grants 438975/2018-9 and 309110/2021-1.
C. Valls is supported through CAMGSD, IST-ID, projects UIDB/04459/2020 and UIDP/04459/2020.

\section*{Data availability statement} Data sharing is not applicable to this article as no new data were created or analyzed in this study.

\end{document}